\documentclass[12pt, reqno]{amsart}
\usepackage[english]{babel}
\usepackage{amsmath,amsfonts,amssymb,amsthm}
\usepackage{hyperref}
\usepackage{graphicx}
\usepackage{color}
\usepackage{caption}
\usepackage{subcaption}
\usepackage{longtable}

\usepackage[margin=2.5cm]{geometry}

\newtheorem{thm}{Theorem}[section]
\newtheorem{lem}[thm]{Lemma}
\newtheorem{cor}[thm]{Corollary}

\theoremstyle{definition}
\newtheorem{defin}[thm]{Definition}
\newtheorem{exam}[thm]{Example}
\theoremstyle{remark}
\newtheorem*{rem}{Remark}

\newcommand { \ib }[1] {\textit{\textbf{#1}}}

\newcommand{\R}{\mathbb{R}}

\newcommand{\Z}{\mathbb{Z}}
\newcommand{\N}{\mathbb{N}}

\begin{document}
\renewcommand{\ib}{\mathbf}
\renewcommand{\proofname}{Proof}
\renewcommand{\phi}{\varphi}
\newcommand{\conv}{\mathrm{conv}}
\newcommand{\tr}{\mathrm{tr}}

\title{On triangular paperfolding patterns}
\author{Alexey~Garber}

\address{School of Mathematical \& Statistical Sciences, The University of Texas Rio Grande Valley, 1 West University blvd., Brownsville, TX, 78520, USA.}
\email{alexeygarber@gmail.com}

\date{\today}

\begin{abstract}
We introduce patterns on a triangular grid generated by paperfolding operations. We show that in case these patterns are defined using a periodic sequence of foldings, they can also be generated using substitution rules and compute eigenvalues and eigenvectors of the corresponding matrices. We also prove that densities of all basic triangles are equal in these patterns.
\end{abstract}

\maketitle

\section{Introduction}

A classical paperfolding sequence can be constructed by taking a strip of paper of length $2^n$ and folding it $n$ times at the center. If we then unfold the resulting strip of length $1$ back into the strip of length $2^n$, then we can read a sequence of valleys and peaks. The one-sided paperfolding sequence is defined as taking the limit as $n$ goes to infinity of the pattern read from the left (see \cite[{\tt A014577}]{oeis}). The two-sided paperfolding sequence is defined as the same pattern read from the center of the strip again as $n$ goes to infinity (see \cite[Sect. 4.5]{BG}).

The properties of the paperfolding sequences have been studied by many authors \cite{All92,AM, DK, DMP, Gar67}. In particular, the paper \cite{DMP} by Dekking, Mend\`es-France, and van der Poorten exhibits a set of substitution rules to generate the paperfolding sequence.

The substitution sequence was generalized to the higher-dimensional case by Ben-Abraham, Quandt and Shapiraa \cite{BQS}. The paper \cite{GN} by G\"ahler and Nilsson provides substitution rules for the corresponding paperfolding structures and studies various properties of the corresponding patterns.

In this paper we study a family of patterns that can be generated using a paperfolding approach. The patterns are formed by peaks and valleys on a triangular grid and obtained by repeated folding of a regular triangle along its midsegments. On each stage the foldings are done through upper or lower halfspace of the ambient $3$-dimensional space, but the choice of upper or lower halfspace is independent for different stages in general.

The main results of the paper are the following. In Theorem \ref{thm:sub} we prove that if we always perform the foldings through the upper halfspace, then the resulting pattern can be generated using substitution rules. In Theorem \ref{thm:periodic} we generalize this result for the case of periodic choices between upper and lower halfspaces. In Theorem \ref{thm:eigensystem} we obtain all eigenvalues and eigenvectors for the corresponding substitution matrices. Finally, in Theorem \ref{thm:densities} we prove that for any choice of halfspaces (not necessarily periodic) densities of all types of unit triangles constituting the corresponding pattern are equal.

The paper is organized as follows. Section \ref{sec:bano} contains basic definitions of triangular folding patterns and related triangular tilings.

Section \ref{sec:lemmas} contains preliminary lemmas that describe the structure of these patterns.

In Section \ref{sec:allup} we prove that the triangular folding pattern defined by the all-up sequence of elementary foldings can be defined using substitution rules. Section \ref{sec:periodic} is devoted to the proof of existence of substitution rules for any periodic sequence of elementary foldings. Also in Section \ref{sec:periodic} we obtain eigenvalues and eigenvectors of substitution matrices for folding patterns defined using periodic sequences of elementary foldings.

In Section \ref{sec:density} we prove that for any sequence of elementary foldings (periodic or not) the densities of different types of triangles present in the corresponding triangular tiling exist and are equal. The types of triangles here are defined using Definition \ref{def:8colors}.

\section{Basic notions}\label{sec:bano}

\begin{defin}
We will work only with regular triangles of a fixed triangular lattice (or grid) $\mathcal{L}$ obtained from a grid of horizontal lines and its rotations by $\pi/3$ and $2\pi/3$ (see Figure \ref{pict:lattice}). We assume that the smallest triangle in this grid has side length $1$, and we call every triangle of $\mathcal L$ with side length $1$ a {\it unit triangle}.

\begin{figure}[!ht]
\begin{center}
\includegraphics[scale=0.7]{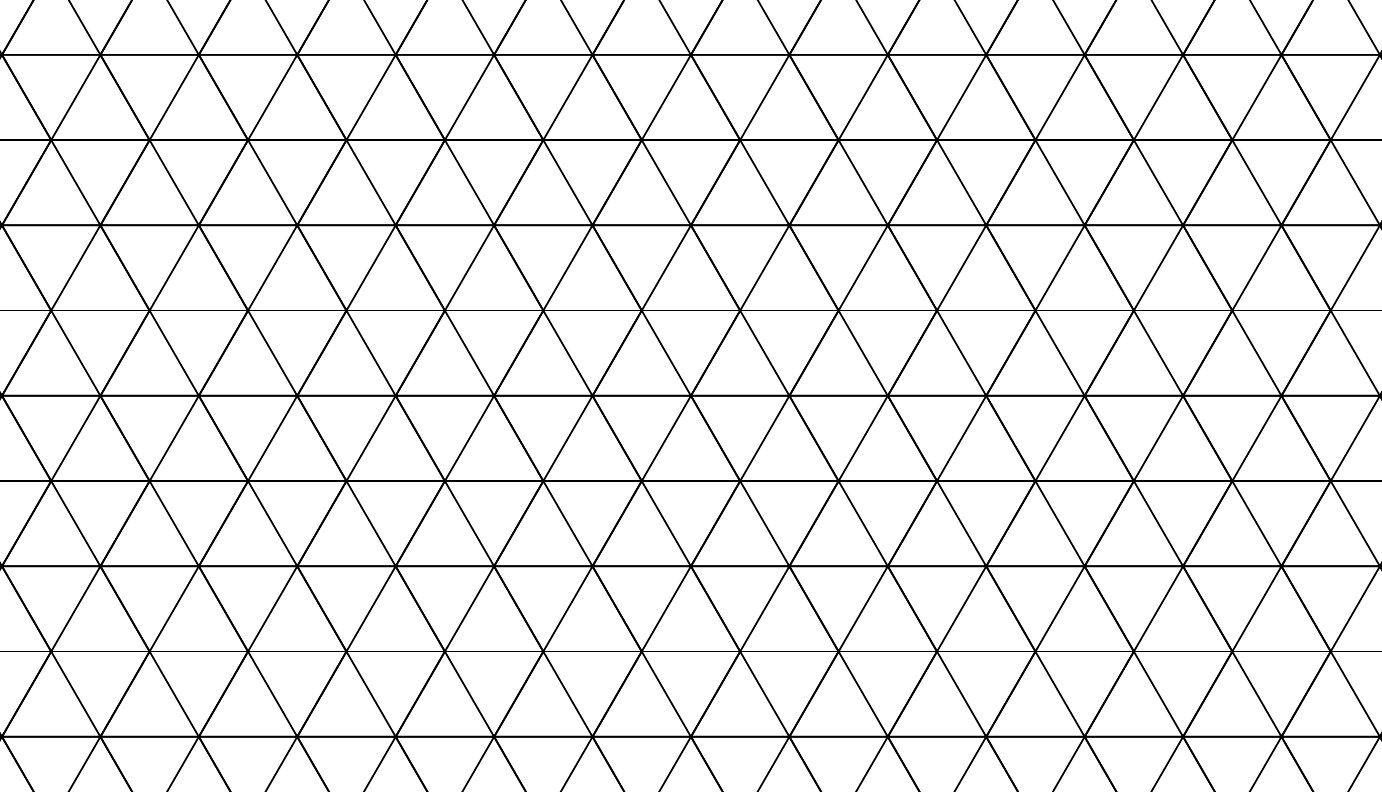} 
\caption{The triangular grid $\mathcal{L}$.}
\label{pict:lattice}
\end{center}
\end{figure}

A regular triangle with integer side length $n\geq 1$ of the grid $\mathcal{L}$ is called {\it positive} if its third vertex is higher than its horizontal side and it is called {\it negative} otherwise, see Figure \ref{pict:triangle}.

\begin{figure}[!ht]
\begin{center}
\includegraphics[scale=0.7]{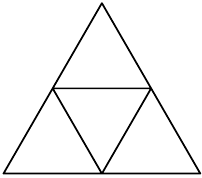}  
\hskip 2cm
\includegraphics[scale=0.7]{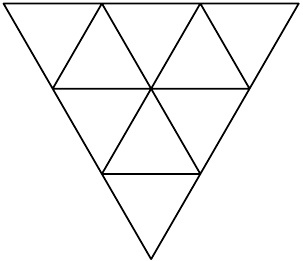}  
\caption{A positive triangle with side length $2$ and a negative triangle with side length $3$.}
\label{pict:triangle}
\end{center}
\end{figure}
\end{defin}

\begin{defin}\label{def:elementary}
By an {\it elementary triangular folding} we mean a folding of a regular triangle, positive or negative, with an even side length $2a$ into a regular triangle, negative or positive respectively, with side length $a$ (covered four times) by folding it in the midsegments (see Figure \ref{pict:folding}).

In general, the separate foldings in each midsegment are independent, so each can be done through the upper or through the lower halfspace of the ambient three-dimensional space. However, in this paper we study only the elementary foldings with all foldings in each midsegment done through the same halfspace. If all three foldings are done through the upper halfspace, then we call the elementary folding a {\it folding up} and encode it with ``$+$''. Alternatively we call the elementary folding a {\it folding down} if all three foldings are done through the lower halfspace and encode such elementary foldings with ``$-$''.
\end{defin}

\begin{figure}[!ht]
\begin{center}
\includegraphics[scale=0.7]{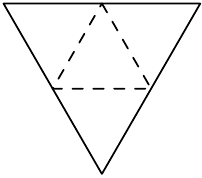}   
$\longrightarrow$
\includegraphics[scale=0.7]{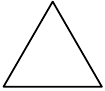}   
\caption{Elementary triangular folding.}
\label{pict:folding}
\end{center}
\end{figure}

In the sequel, the dashed triangle of side length $a$ in the left part of Figure \ref{pict:folding} is called the {\it central part} of the initial triangle with side length $2a$, and the other three triangles with side length $a$ (each with one dashed and two solid sides) are called the {\it side parts} of the initial triangle with side length $2a$.

If we unfold the four-layered triangle with side length $a$ back into the triangle with side length $2a$, then the midsegments will form peaks if the elementary folding was a folding down, or valleys if the elementary folding was a folding up. In the former case we will color the midsegments with blue (peaks), and in the latter case with red (valleys).

Suppose $\mathcal{F}=\{a_1,\ldots,a_k\}$ is a sequence of elementary foldings, so each $a_i$ is either $+$ or $-$ and encodes whether the corresponding folding should be done through upper or lower halfspace forming valley or peaks respectively.

We construct the {\it folding pattern} corresponding to $\mathcal{F}$ in the following way. We take a regular triangle $T$ (made from paper) with side length $2^k$ which is positive if $k$ is even and negative if $k$ is odd. Then we fold $T$ according to the elementary foldings in $\mathcal{F}$ performing $a_k$ first, then $a_{k-1}$, and so on, finishing with $a_1$. In the end we get a positive triangle with side length $1$ consisting of $4^k$ layers (of paper).

If we unfold the unit triangle back to its initial size, then every unit segment of the grid $\mathcal{L}$ in the interior of  $T$ will become a valley or a peak. If we unfold a triangle with side length $2^n (n<k)$ that has a folding pattern of peaks and valleys, then the resulting triangle with side length $2^{n+1}$ will have the same pattern inside its central part, but the patterns in the side parts will be different because each valley will become a peak after a single unfolding, and each peak will become a valley. We will describe this dependence in more detail later.

\begin{defin}
The {\it folding pattern} $P_\mathcal{F}$ corresponding to $\mathcal{F}$ is the coloring of all peaks in the resulting unfolded triangle $T$ with blue, and all valleys with red. An example of the folding pattern corresponding to $\mathcal{F}=\{a_1,\ldots,a_6\}$ where each $a_i$ is done through the upper halfspace, or $\mathcal{F}=\{+,+,+,+,+,+\}$ for short, is shown in Figure \ref{pict:allup-6}.

\begin{figure}[!ht]
\begin{center}
\includegraphics[width=\textwidth]{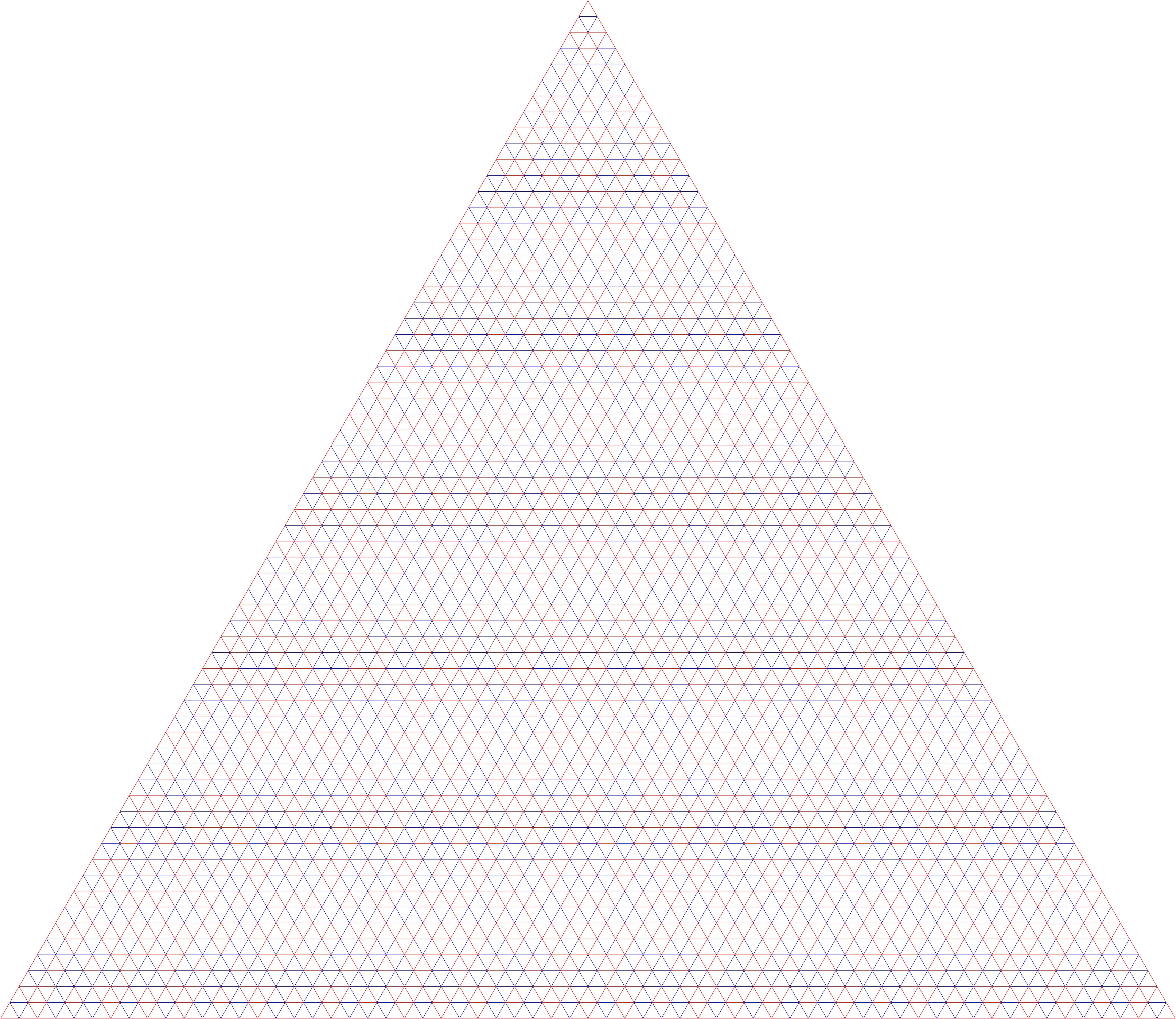}   
\caption{The folding pattern for sequence of six elementary foldings up (for a certain choice of boundary).}
\label{pict:allup-6}
\end{center}
\end{figure}
\end{defin}

If a sequence $\mathcal{F}'=\{a_i\}_{i=1}^m$ is a subsequence of $\mathcal{F}=\{a_i\}_{i=1}^{m+1}$, then the pattern $P_{\mathcal{F}'}$ is a subpattern of $P_\mathcal{F}$. Namely, the (interior of the) central triangle with side length $2^m$ of $P_\mathcal{F}$ will be colored according to the pattern $P_{\mathcal{F}'}$. This allows us to introduce a coloring of the grid $\mathcal{L}$ as the pattern corresponding to an infinite sequence of elementary foldings.

\begin{defin}\label{def:pattern}
Let $\mathcal{S}=\{a_i\}_{i=1}^\infty$ be a sequence of elementary foldings. We fix a positive unit triangle $T_0$ of $\mathcal{L}$ and define $P_\mathcal{S}$ as the limit of patterns $P_{\mathcal{S}_n}$ where $\mathcal{S}_n=\{a_i\}_{i=1}^n$ and each $P_{\mathcal{S}_n}$ has $T_0$ as its central triangle, i.e. the unique triangle of $\mathcal L$ that contains the center of $P_{\mathcal S_n}$.

The limit is naturally defined and can be treated as the limit in local topology (see Definition \ref{def:hull} below). Each unit segment of the grid $\mathcal{L}$ belongs to the interior of $(-2)^k$-dilation of $T_0$ for some $k$, so this unit segment is colored with blue or red in the pattern $P_{\mathcal{S}_k}$. Moreover, the color of this segment does not depend on the choice of appropriate $k$ as sequences $\mathcal{S}_k$ and $\mathcal{S}_m$ for $m>k$ give the same coloring of all common segments of $\mathcal{L}$ except the boundary of $P_{\mathcal S_k}$.
\end{defin}

\begin{defin}\label{def:hull}
Using the pattern $P_\mathcal{S}$ we can define the {\it hull $H_\mathcal{S}$} corresponding to an infinite sequence $\mathcal{S}$ as a topological closure of the set of $\R^2$-translations of $P_\mathcal S$ in the local topology. In this topology, two patterns are said to be $\varepsilon$-close if we can shift one of these patterns by a vector of length at most $\varepsilon$, and the patterns will coincide in the $\frac{1}{\varepsilon}$-ball centered at the origin. It is clear that the resulting hull does not depend on the choice of $T_0$. We refer to \cite[Sect. 5.4]{BG} for a more detailed definition of the geometric hull associated with a given pattern.
\end{defin}

The pattern $P_\mathcal{S}$ is given as a coloring of the grid $\mathcal{L}$, however we can transform it into a tiling of the plane with decorated regular triangles of eight types defined by one of four colors and positivity/negativity. Namely, we have four colors for positive triangles based on the number of red sides from  0 to 3 with decoration for the triangles with 1 or 2 red sides opposite to the only red or blue side of the triangle, see Figure \ref{pict:colored_tr}. Similarly we have four colors for negative triangles, see Figure \ref{pict:colored_tr} as well.

\begin{figure}[!ht]
\begin{center} 
\begin{tabular}{lll}
\includegraphics[scale=0.7]{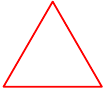}$\longrightarrow$\includegraphics[scale=0.7]{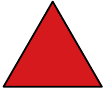}&\hskip 3cm&\includegraphics[scale=0.7]{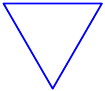}$\longrightarrow$\includegraphics[scale=0.7]{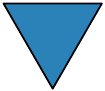}\\
\includegraphics[scale=0.7]{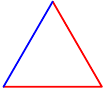}$\longrightarrow$\includegraphics[scale=0.7]{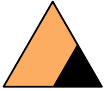}&\hskip 3cm&\includegraphics[scale=0.7]{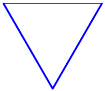}$\longrightarrow$\includegraphics[scale=0.7]{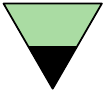}\\
\includegraphics[scale=0.7]{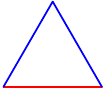}$\longrightarrow$\includegraphics[scale=0.7]{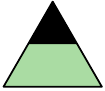}&\hskip 3cm&\includegraphics[scale=0.7]{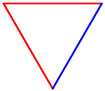}$\longrightarrow$\includegraphics[scale=0.7]{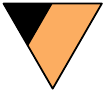}\\
\includegraphics[scale=0.7]{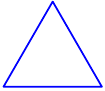}$\longrightarrow$\includegraphics[scale=0.7]{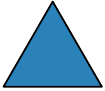}&\hskip 3cm&\includegraphics[scale=0.7]{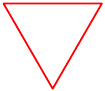}$\longrightarrow$\includegraphics[scale=0.7]{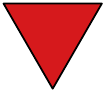}\\
\end{tabular}
\caption{The rules for constructing the tiling $T_\mathcal{S}$ from the pattern $P_\mathcal{S}$.}
\label{pict:colored_tr}
\end{center}
\end{figure}

\begin{defin}\label{def:8colors}
We will call the resulting tiling $T_\mathcal{S}$ the {\it folding tiling} associated with $\mathcal{S}$.
\end{defin}

For example, the tiling given in Figure \ref{pict:allup-6} transformed into the tiling with triangles of four colors is given in Figure \ref{pict:allup6-color}. We removed the decoration in this figure; as we show later the decoration is redundant.

\begin{figure}[!ht]
\begin{center} 
\includegraphics[width=\textwidth]{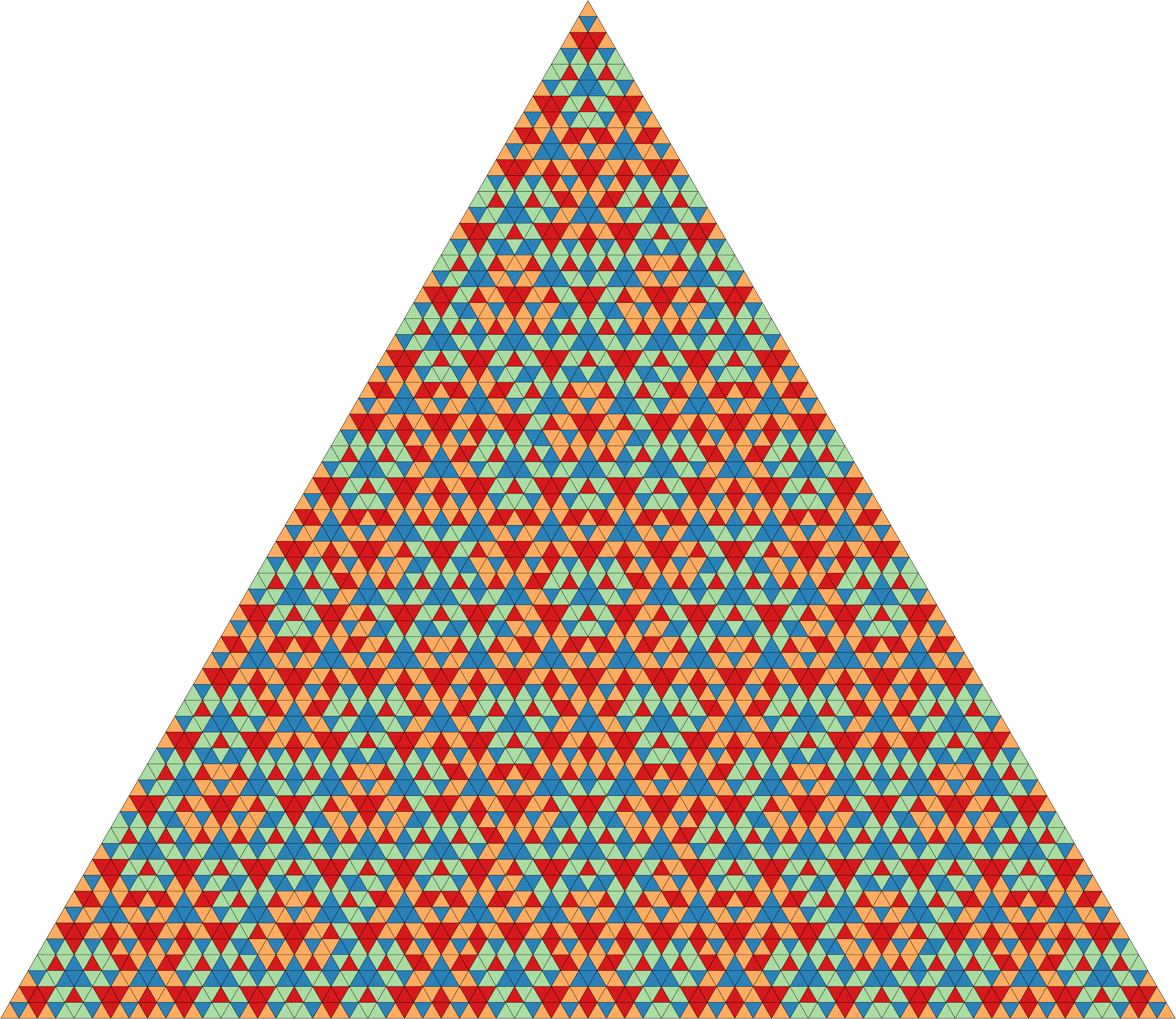}
\caption{The tiling with triangles of four colors without decorations corresponding to the sequence of six elementary foldings up (for a certain choice of boundary).}
\label{pict:allup6-color}
\end{center}
\end{figure}

It is clear that the pattern $P_\mathcal S$ and the tiling $T_\mathcal S$ can be obtained from each other using local information only. This is summarized in the following (informal) definition of MLD tilings, see also \cite{BSJ}. We can use this definition for patterns on $\mathcal L$ as well.

\begin{defin}
A tiling $T$ is {\it locally derivable} from a tiling $S$ (with radius $r$), if local equality (congruency or translation equivalence) of $r$-balls of $S$ at two points $x$ and $y$ implies equality of local structures of $T$ at $x$ and $y$.

Two tilings $S$ and $T$ are {\it MLD} ({\it mutually locally derivable}), if $S$ is locally derivable from $T$ and vice versa. We refer to \cite[Sect. 5.2]{BG} for more details. 
\end{defin}

So, the pattern $P_\mathcal{S}$ and the tiling $T_\mathcal{S}$ belong to the same MLD class. Therefore, the corresponding hulls are MLD as well. Later we will show that removal of the decoration on triangles preserves the MLD class of a folding tiling.

\section{Preliminary lemmas}\label{sec:lemmas}

Recall that $T_0$ is the positive unit triangle we use as our ``reference point'', that is, if we construct a folding pattern for a finite sequence, then $T_0$ is exactly the triangle we get after performing all elementary foldings. Our first goal is to split $\mathcal L$ into a disjoint union of ``layers'' that are dilations of each other. As a simple description, we take the set given in Figure \ref{pict:layer}, and dilate it with factors $(-2)^k$ with respect to the center of the shaded triangle. Below is a rigorous description of such sets and their properties.

Let $O$ denote the center of $T_0$. Let $\xi_1$ be the vertical vector of length $2\sqrt3$ pointing downwards and let $\xi_2$ and $\xi_3$ be counterclockwise rotations of $\xi_1$ by $\frac{2\pi}{3}$ and $\frac{4\pi}{3}$ respectively. Then the grid $\mathcal{L}$ can be written as the following disjoint union of lines

$$\mathcal{L}=\bigsqcup\limits_{i=1}^3\bigsqcup\limits_{n\in\Z}\ell_{i,n}$$
where $\ell_{i,n}=\{\mathbf{x}\in\R^2|(\mathbf{x},\xi_i)=3n+1\}$ and $(\cdot,\cdot)$ denotes the standard inner product in $\R^2$.

We split the grid $\mathcal L$ into layers $\mathcal L_k$ for $k\in \N$ according to the following definition. Let $k$ be a natural number. Let $$\Lambda_k:=2^k\Z+\frac{1}{3}((-2)^{k-1}-1).$$ The set $\Lambda_k$ is exactly the set of all integers $n$ such that the largest power of $2$ that divides $3n+1$ is $2^{k-1}$ so $\Z$ is split into the disjoint union $$\Z=\bigsqcup\limits_{k\in \N}\Lambda_k.$$ 

\begin{defin} 
We define the $k$th {\it layer} $\mathcal{L}_k$ of $\mathcal{L}$ as the union 
$$\mathcal{L}_k=\bigsqcup\limits_{i=1}^3\bigsqcup\limits_{n\in \Lambda_k}\ell_{i,n}$$ so the line $\ell_{i,n}$ belongs to the layer $\mathcal{L}_k$ if and only if the largest power of $2$ that divides $3n+1$ is $2^{k-1}$.

Then the grid $\mathcal L$ is split into the union $$\mathcal{L}=\bigsqcup\limits_{k\in\mathbb{N}}\mathcal{L}_k.$$
\end{defin}

The layer $\mathcal{L}_1$ is shown in Figure \ref{pict:layer}. 
\begin{figure}[!ht]
\begin{center}
\includegraphics[width=0.8\textwidth]{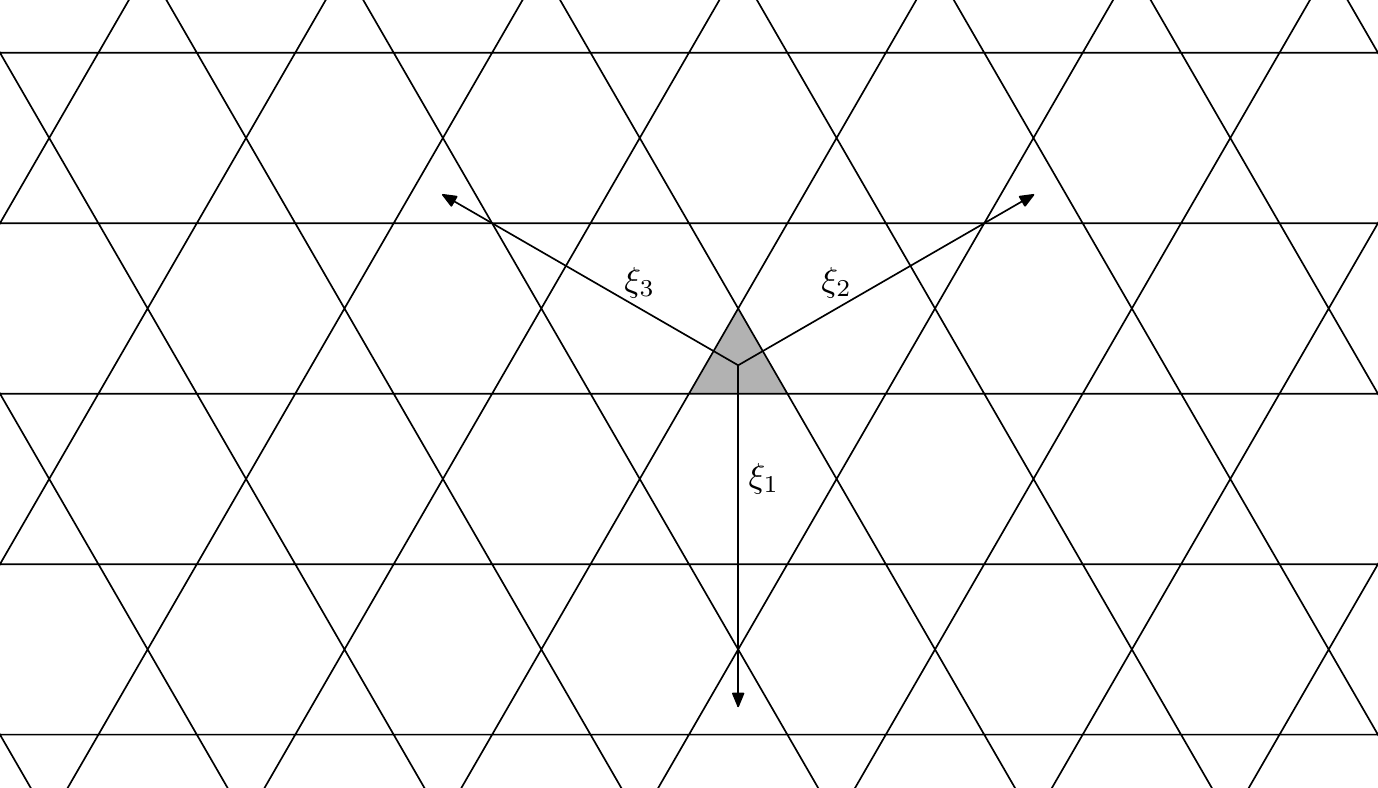}
\caption{The layer $\mathcal{L}_1$ with shaded triangle $T_0$ and the vectors $\xi_i$.}
\label{pict:layer}
\end{center}
\end{figure}

We also observe that $\mathcal{L}_{k+1}$ is the $(-2)$-dilation of $\mathcal{L}_k$ with respect to $O$. Indeed, the line $\ell_{i,n}$ is given by equation $(\mathbf{x},\xi_i)=3n+1$. After $(-2)$-dilation we get the line $(\mathbf{x},\xi_i)=-2(3n+1)=-6n-2=3(-2n-1)+1$ or the line $\ell_{i,-2n-1}$. If $\ell_{i,n}\in \mathcal{L}_k$, then the largest power of $2$ that divides $3n+1$ is $2^{k-1}$ and therefore the largest power of $2$ that divides $3(-2n-1)+1=-2(3n+1)$ is $2^k$ and $\ell_{i,-2n-1}\in \mathcal L_{k+1}$.

From Figure \ref{pict:layer} we can see that $\mathcal{L}_1$ can be viewed as a collection of triangles with side length $1$. Similarly, since the layer $\mathcal{L}_2$ is a $(-2)$-dilation of $\mathcal{L}_1$, it can be viewed as a collection of triangles with side length $2$. In general, the layer $\mathcal{L}_k$ is a collection of triangles with side length $2^{k-1}$.

\begin{defin}\label{def:tr+hex}
We will call these triangles with side length $2^{k-1}$ the {\it positive} or {\it negative triangles of the layer $\mathcal{L}_k$}.

The remaining regular hexagons with side length $2^{k-1}$ are called the {\it hexagons of the layer $\mathcal{L}_k$}.
\end{defin}

\begin{lem}\label{lem:layer}
Let $\mathcal{S}=\{a_i\}_{i=1}^\infty$ be an infinite sequence of elementary foldings, and let $P_\mathcal{S}$ be the corresponding folding pattern. Then for every $k\in \mathbb{N}$ the coloring of $\mathcal{L}_k$ in $P_\mathcal{S}$ is defined by $a_k$ only. 

Particularly, if $k$ is odd and $a_k=+$  or if $k$ is even and $a_k=-$, then positive triangles of $\mathcal{L}_k$ are colored with red and negative triangles of $\mathcal{L}_k$ are colored with blue. Alternatively, if $k$ is odd and $a_k=-$ or if $k$ is even and $a_k=+$, then positive triangles of $\mathcal{L}_k$ are colored with blue and negative triangles of $\mathcal{L}_k$ are colored with red.
\end{lem}

\begin{proof}
We will prove by induction that for any $n\geq k$ the intersection of $\mathcal{L}_k$ with the triangle with side length $2^n$ centered at $O$ (positive if $n$ is even, negative if $n$ is odd) is colored according to the description in the statement.

The basis of induction for $n=k$ is obvious because in this case the pattern inside the triangle with side length $2^k$ is exactly the pattern $P_{\mathcal{S}_k}$ defined by the sequence $\mathcal{S}_k=\{a_i\}_{i=1}^k$, and the elementary folding $a_k$ defines only the coloring of the midsegments of this triangle. The midsegments are colored with red if $a_k$ is a folding up, and with blue if $a_k$ is a folding down. These midsegments form a positive triangle if $k$ is odd and a negative triangle if $k$ is even.

The induction step $n\longrightarrow n+1$ can be shown by analyzing an additional unfolding of the triangle with side length $2^n$ into a triangle with side length $2^{n+1}$. The pattern inside the central triangle stays the same and it is also unfolded in each of three side triangles, see Figure \ref{pict:layercol} for an illustration with $k=1$ and $n=3$.

In order to construct the pattern in the three side triangles with side length $2^n$ we need to perform reflections in midsegments of the triangle with side length $2^{n+1}$ and change the colors of the new segments because each of the unfolded peaks in a side triangle becomes a valley and vice versa. Also each positive triangle of $\mathcal{L}_k$ unfolds into a negative triangle of $\mathcal{L}_k$ and each negative triangle of $\mathcal{L}_k$ unfolds into a positive triangle of $\mathcal{L}_k$. Indeed, without loss of generality we can assume only unfolding with respect to the horizontal midsegment; (the line that contains) this midsegment has equation $(\mathbf x,\xi_1)=(-2)^{n}$,  we denote this line as $\ell$. Suppose we reflect the line $\ell_{i,m}\in \mathcal L_k$ with equation $(\mathbf x,\xi_i)=3m+1$ about $\ell$. If $i=1$, then the reflection has equation $(\mathbf x,\xi_1)=-3m-1+2(-2)^n$ and the largest power of $2$ that divides the right-hand side is $2^{k-1}$ (same as for $3m+1$) because $k<n+1$, so this line belongs to $\mathcal L_k$ as well.

For two other options, we assume that $i=2$ (the case $i=3$ is similar), then the reflection of $\ell_{2,m}$ about $\ell$ has equation $(\mathbf x,\xi_3)=-3m-1-(-2)^n$. Indeed, the reflection has equation $(\mathbf x,\xi_3)=\alpha$ for some $\alpha$. Since the left-hand sides of equations for $\ell$, $\ell_{2,m}$ and its reflection add up to $0$, then the right-hand sides must add to $0$ as well as all three lines have one common point. This implies $\alpha = -3m-1-(-2)^n$. 

The line with equation $(\mathbf x,\xi_3)=-3m-1-(-2)^n$ belongs to $\mathcal L_k$ as well. Therefore, all sides of a triangle from $\mathcal L_k$ are reflected into lines of $\mathcal L_k$, so the reflected triangle is also from $\mathcal L_k$. It is clear that each negative triangle is reflected into a positive triangle, and each positive triangle is reflected into a negative triangle. Moreover, the colors have to be swapped as each peak is unfolded into a valley and vice versa.

Together these two observations (change of color and change of positiveness of triangle) complete the induction step.

\begin{figure}[!ht]
\begin{center} 
\includegraphics[width=0.8\textwidth]{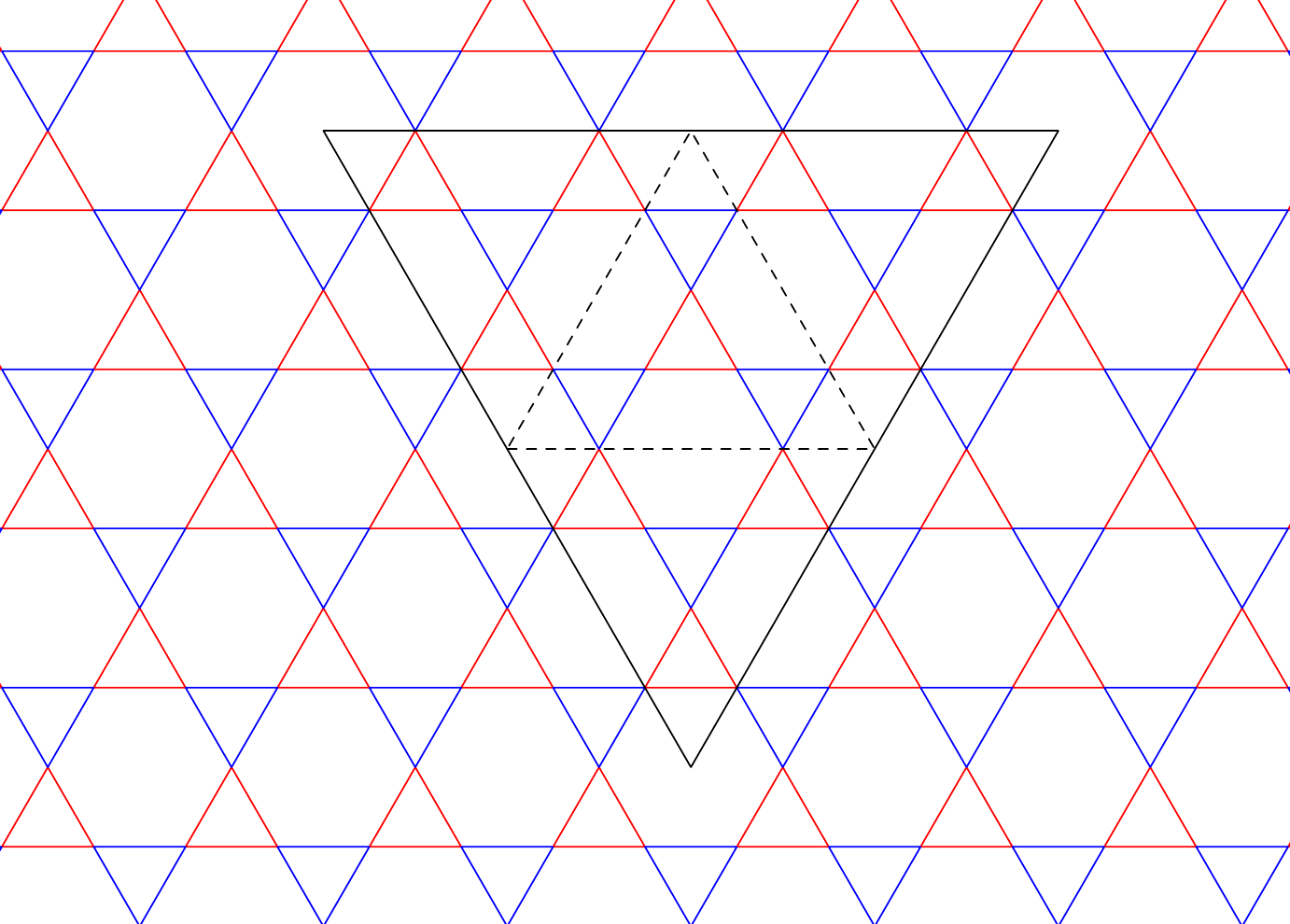}
\caption{A coloring of the layer $\mathcal{L}_1$ if $a_1$ is a folding up. The triangles with side length 8 (solid lines) and 4 (dashed lines) are shown to illustrate the induction step from Lemma \ref{lem:layer}.}
\label{pict:layercol}
\end{center}
\end{figure}
\end{proof}

The next lemma describes all possible local colorings at a vertex of the grid $\mathcal L$.

\begin{lem}\label{lem:hexagon}
Let $\mathcal S$ be an infinite sequence of elementary foldings. Among six segments at a vertex of the pattern $P_\mathcal{S}$, two consecutive segments are colored with one color (red or blue) and four others are colored with the the other color (blue or red respectively), see Figure $\ref{pict:local}$.
\end{lem}
\begin{proof}
First of all, we claim that a vertex $v$ of $\mathcal{L}$ is a vertex of two triangles of some layer $\mathcal{L}_k$. Indeed, it is true if a vertex belongs to $\mathcal{L}_1$, and if it doesn't, then it is in the union $\bigsqcup\limits_{k=2}^\infty\mathcal{L}_k=(-2)\mathcal{L}$, and we can repeat these arguments. 

Therefore four of six segments at $v$ are segments of $\mathcal{L}_k$ and they form two pairs of opposite segments such that two neighbor segments (two segments with angle $\frac{\pi}{3}$ between them) are colored with red, and two segments opposite to these red segments are colored with blue. The remaining two segments are opposite and belong to a side of triangle of some layer $\mathcal{L}_n$ for $n>k$, therefore they are colored both with red or both with blue. Both these cases are exactly the cases from Figure \ref{pict:local}.
\end{proof}

\begin{figure}[!ht]
\begin{center} 
\includegraphics[scale=0.7]{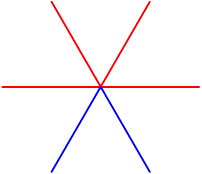}
\hskip 0.15\textwidth
\includegraphics[scale=0.7]{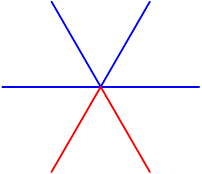}
\caption{Two possible arrangements at a vertex of $\mathcal{L}$.}
\label{pict:local}
\end{center}
\end{figure}

\begin{defin}
Let $P$ be a pattern in $\R^2$. We say that $P$ is {\it non-periodic} if the equality $P+\mathbf{t}=P$ for a vector $\mathbf{t}\in\R^2$ implies $\mathbf t=0$.

Any vector $\mathbf{t}$ that satisfies $P+\mathbf t=P$ is called a {\it period} of $P$.
\end{defin}

\begin{thm}\label{thm:aperiodic}
For every sequence $\mathcal{S}=\{a_i\}_{i=1}^\infty$ of elementary foldings, the pattern $P_\mathcal{S}$ is non-periodic.
\end{thm}
\begin{proof}
Suppose $\mathbf t$ is a period of $P_\mathcal{S}$. We note that each layer $\mathcal{L}_k$ is a collection of lines and that each line of $\mathcal L_k$ is split into blue and red segments of length exactly $2^{k-1}$ in an alternating way. Thus $\mathcal{L}_k+\mathbf{t}=\mathcal{L}_k$ because no line from the translation $\mathcal{L}_k+\mathbf{t}$ can belong to a layer other than $\mathcal{L}_k$ since a segment of length $3\cdot 2^{k-1}$ with the middle third colored with blue and two other thirds colored with red must belong to $\mathcal{L}_k$. Thus $\mathbf{t}$ is a period for every layer $\mathcal{L}_k$.

It is easy to see that every non-zero period of $\mathcal{L}_k$ has length at least $2^k$, thus if $\mathbf t\neq 0$, then $|\mathbf t|\geq 2^k$ for every $k$ which is impossible.
\end{proof}

\begin{rem}
The pattern $P_\mathcal{S}$ is {\it limit periodic} meaning it can be represented as a disjoint union of periodic patterns, the layers $\mathcal{L}_k$, with commensurate periods. We refer to \cite{BG} (Section 4.5 in particular) for more examples of tilings and patterns that are limit periodic and properties of such tilings. We also refer to \cite{FH} for another example of limit periodic pattern that shares many properties with the patterns under study.
\end{rem}

\begin{cor}
The tiling $T_\mathcal{S}$ as well as the hull $H_\mathcal{S}$ are aperiodic.
\end{cor}
\begin{proof}
Recall that a hull is aperiodic if it does not contain a periodic pattern/tiling (see \cite[Def. 5.12]{BG}, for example). If $P$ is a  pattern from  $H_\mathcal{S}$, then we can use the same arguments to find a contradiction as in the proof of Theorem \ref{thm:aperiodic} as the arguments use only local properties of $P_\mathcal{S}$.

For the tiling $T_\mathcal{S}$ we use that it is MLD equivalent to $P_\mathcal{S}$ (see the proof of Theorem \ref{thm:mld} below), and therefore is aperiodic.
\end{proof}

Now we will show that the decoration of the tiling $T_\mathcal{S}$ can be reconstructed from the undecorated tiling.

\begin{defin}
Let $T'_\mathcal{S}$ denote the {\it undecorated folding tiling} corresponding to an infinite sequence $\mathcal{S}$ of elementary foldings. The tiling $T'_\mathcal{S}$ is obtained from $T_\mathcal{S}$ by removing the decoration of triangles with both red and blue sides.
\end{defin}

It is worth noting that the tiling $T_\mathcal{S}$ has 16 translation types of different triangles, and the tiling $T'_\mathcal{S}$ has only 8 different translation types.

\begin{thm}\label{thm:mld}
The tilings $T_\mathcal{S}$ and $T'_\mathcal{S}$ can be locally reconstructed from each other, that is they belong to the same MLD equivalence class of tilings.
\end{thm}

\begin{proof}
The tiling $T'_\mathcal{S}$ can be reconstructed from $T_\mathcal{S}$ by removing decoration, and this is clearly a local operation. 

In order to reconstruct $T_\mathcal{S}$ from $T'_\mathcal{S}$ we suppose that we are given a tiling $T'$ that is constructed as $T'_\mathcal{S}$ for some (unknown) sequence $\mathcal{S}$ of elementary foldings, and our goal is to construct the corresponding unique tiling pattern $P_\mathcal{S}$ using local patches of $T'$. After that the tiling $T_\mathcal{S}$ can be constructed from the pattern $P_\mathcal{S}$ using the definition of a folding tiling.

All triangles of the layer $\mathcal{L}_1$ have sides of one color, so the coloring of their sides can be reconstructed from the tiling $T'$. Also the layer $\mathcal{L}_1$ can be identified as well. Indeed, if there is a segment of length $3$ colored in alternating pattern red-blue-red or blue-red-blue, then this segment belongs to $\mathcal{L}_1$ because all other layers must contain segments of length at least $2$ of a single color. Note, that since the length of the segment we use is bounded, this process is local. Thus, the only part left is to reconstruct the coloring inside the hexagons of $\mathcal{L}_1$.

Let $H$ be such a hexagon, then its sides are colored alternatively with red and blue and we can reconstruct their colors. For six triangles inside $H$ the tiling $T'$ gives us information about the number of red and blue sides for each. Therefore, if we will identify the color of at least one segment inside $H$, then all six will be recovered.

Now we refer to Figure \ref{pict:local}. We can see that from six segments at the center of $H$, there are three (actually four) consecutive segments of one color. Since the sides of $H$ have alternating colors, there will be a triangle inside $H$ with all sides of one color, and those sides can be recognized locally from the tiling $T'$ inside $H$ (provided we already identified the layer $\mathcal{L}_1$).

Thus we have described a local algorithm of finding $P_\mathcal{S}$ from $T'$.
\end{proof}

\begin{thm}\label{thm:mldpatterns}
If two sequences $\mathcal{S}$ and $\mathcal{R}$ of elementary foldings differ only in finitely many terms, then the corresponding patterns $P_\mathcal{S}$ and $P_\mathcal{R}$ belong to the same MLD equivalence class.
\end{thm}
\begin{proof}
Suppose that $\mathcal{S}=\{a_i\}_{i=1}^\infty$ and $\mathcal{R}=\{b_i\}_{i=1}^\infty$ and $a_i=b_i$ for $i>n$ for some natural $n$. We reconstruct the color of a fixed segment $X$ in $P_\mathcal{R}$ using its $2^n+1$ (open) neighborhood in $P_\mathcal{S}$.

This neighborhood contains the extension of $T$ by $2^n$ in both directions. If the longest extension colored with single color has length $2^k$ with $k<n$, then $T$ belongs to $\mathcal{L}_{k+1}$. In that case if $a_{k+1}=b_{k+1}$, then we keep the color of $T$, because the layer $\mathcal{L}_{k+1}$ is colored in the same way in $P_\mathcal{S}$ and in $P_\mathcal{R}$ according to Lemma \ref{lem:layer}. If $a_{k+1}\neq b_{k+1}$, then we change the color of $T$ by the same reason.

In case the longest extension of $T$ of a single color has length $2^n$ or greater, then we keep the color of $T$ as in that case $T$ belongs to $\mathcal{L}_{k+1}$ with $k+1>n$ and $a_{k+1}=b_{k+1}$.
\end{proof}

\section{Sequence of all foldings up}\label{sec:allup}

In this section we study the tiling $T_\mathcal{S}$ and the pattern $P_\mathcal{S}$ for the sequence $\mathcal{S}=\{a_i\}_{i=1}^\infty$ where  $a_i=+$ for every $i$. We show that the tiling $T_\mathcal{S}$ is a substitution tiling and the pattern $P_\mathcal{S}$ is a substitution pattern . After that we use the Perron-Frobenius theory for substitution tilings to find densities of all types of tiles in $T_\mathcal{S}$ and all types of unit triangle colorings in $P_\mathcal S$.

Informally, for a given set of polytopes $P_1,\ldots,P_n$ in $\R^d$ and a real number $\lambda>1$, substitution rules are given as dissections of inflated copies $\lambda P_i$ into (translational, congruent) copies of $P_1,\ldots,P_n$. If we repeat the process of inflation and dissection further, then under certain conditions we will be able to construct a tiling of $\R^d$ which is invariant under further inflation and dissection steps. 

We refer to \cite[Ch. 6]{BG} and \cite{FGH} for the exact definition(s) of the substitution tiling/pattern, and to \cite[Sect 2.4]{BG} for more details on Perron-Frobenius theory.

\begin{defin}
We define the ``pattern'' substitution rule $F_P$ and the ``tiling'' substitution rule $F_T$ using  Figure \ref{pict:sub}.

\begin{figure}[!ht]
\begin{center} 
\begin{tabular}{llll}
\includegraphics[scale=0.5]{basic-111.pdf}$\longrightarrow$\includegraphics[scale=0.5]{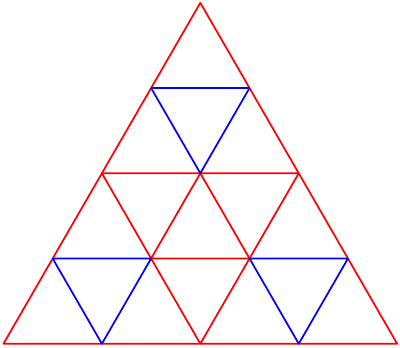}&
\includegraphics[scale=0.5]{basic-112.pdf}$\longrightarrow$\includegraphics[scale=0.5]{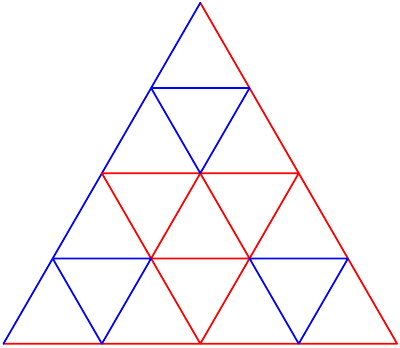}&
\includegraphics[scale=0.5]{basic-122.pdf}$\longrightarrow$\includegraphics[scale=0.5]{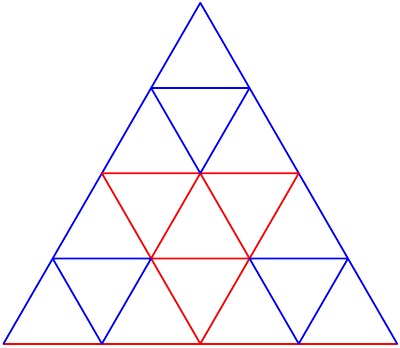}&
\includegraphics[scale=0.5]{basic-222.pdf}$\longrightarrow$\includegraphics[scale=0.5]{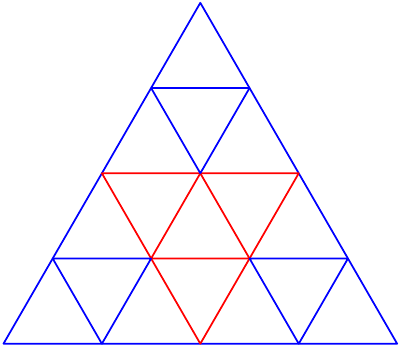}
\\[0.2cm]
\includegraphics[scale=0.5]{basic-444.pdf}$\longrightarrow$\includegraphics[scale=0.5]{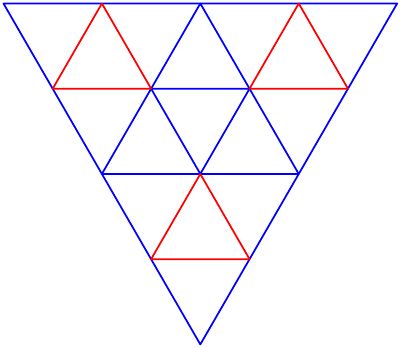}&
\includegraphics[scale=0.5]{basic-344.pdf}$\longrightarrow$\includegraphics[scale=0.5]{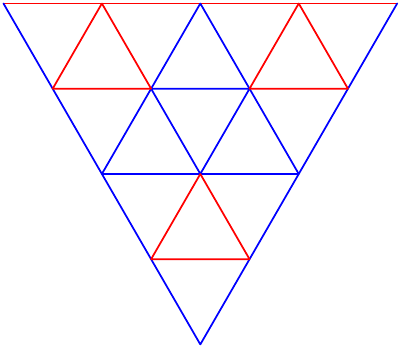}&
\includegraphics[scale=0.5]{basic-334.pdf}$\longrightarrow$\includegraphics[scale=0.5]{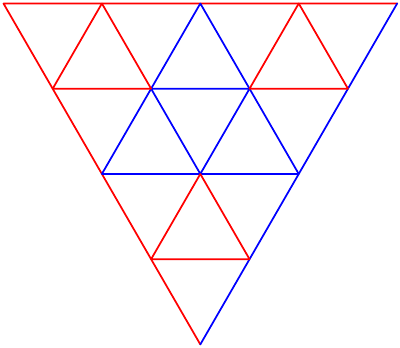}&
\includegraphics[scale=0.5]{basic-333.pdf}$\longrightarrow$\includegraphics[scale=0.5]{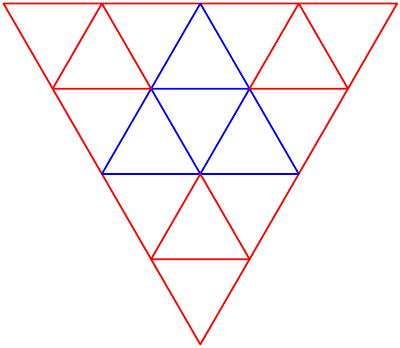}
\\[0.4cm]
\includegraphics[scale=0.5]{basic-555.pdf}$\longrightarrow$\includegraphics[scale=0.5]{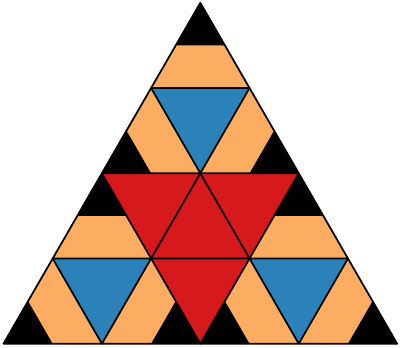}&
\includegraphics[scale=0.5]{basic-556.pdf}$\longrightarrow$\includegraphics[scale=0.5]{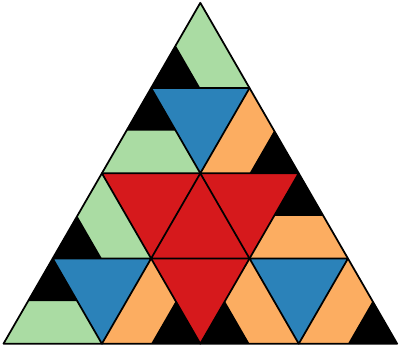}&
\includegraphics[scale=0.5]{basic-566.pdf}$\longrightarrow$\includegraphics[scale=0.5]{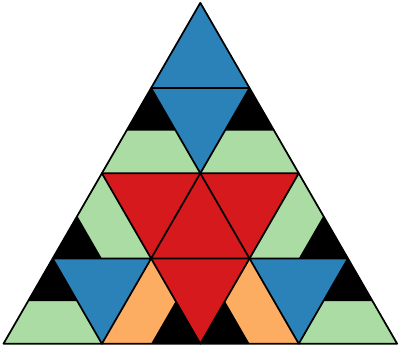}&
\includegraphics[scale=0.5]{basic-666.pdf}$\longrightarrow$\includegraphics[scale=0.5]{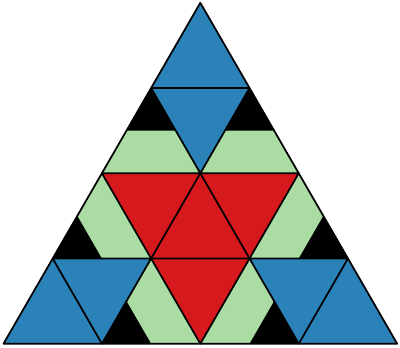}
\\[0.2cm]
\includegraphics[scale=0.5]{basic-888.pdf}$\longrightarrow$\includegraphics[scale=0.5]{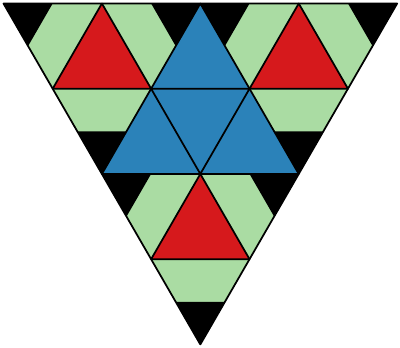}&
\includegraphics[scale=0.5]{basic-788.pdf}$\longrightarrow$\includegraphics[scale=0.5]{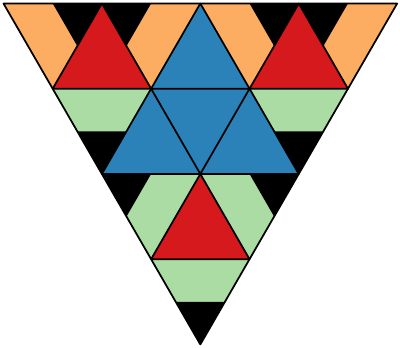}&
\includegraphics[scale=0.5]{basic-778.pdf}$\longrightarrow$\includegraphics[scale=0.5]{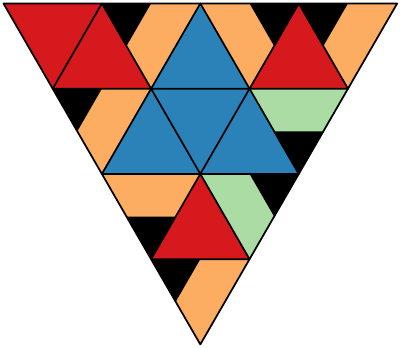}&
\includegraphics[scale=0.5]{basic-777.pdf}$\longrightarrow$\includegraphics[scale=0.5]{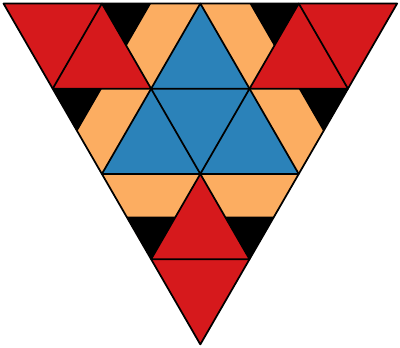}
\end{tabular}
\caption{Substitution rules $F_P$ and $F_T$.}
\label{pict:sub}
\end{center}
\end{figure}
\end{defin}

Note, that these rules are ``derivable'' from each other in a sense, that triangles with side length 4 of $F_T$ can be reconstructed from triangles with side length 4 of $F_P$ using the same approach we use in the proof of Theorem \ref{thm:mld}. We can also use triangles with side length 4 of $F_T$ to reconstruct all segments in triangles with side length 4 of $F_P$ except 6 segments incident to the vertices. These segments incident to vertices can be reconstructed as well because each side of each triangle of side length 4 in $F_P$ is either completely red or completely blue.

\begin{thm}\label{thm:sub}
The substitution rule $F_P$ with a (legal) seed \includegraphics[scale=0.4]{basic-111.pdf} centered at the origin generates the pattern $P_\mathcal{S}$ for the sequence $\mathcal{S}$ of all elementary foldings up.

The substitution rule $F_T$ with a (legal) seed \includegraphics[scale=0.4]{basic-555.pdf} centered at the origin generates the (decorated) tiling $T_\mathcal{S}$ for the sequence $\mathcal{S}$ of all elementary foldings up.
\end{thm}
\begin{proof}
We prove the statement for the pattern substitution rule $F_P$. Then the statement for the tiling substitution rule will follow immediately from the remark before the current theorem.

The rule $F_P$ preserves the coloring on the external sides of a triangle, therefore after applying the rule $F_P$ to the positive triangle of side length $1$ with red sides centered at $O$ $k$ times we get a pattern inside a positive triangle $T_k$ of side length $4^k$ centered at $O$ with red sides. We will prove by induction that the pattern inside $T_k$ obtained via $F_P$ coincides with the pattern $P_\mathcal{S}$ inside the triangle of the same size centered at $O$.

The basis of the induction is evident. For the step $k\longrightarrow k+1$ of the induction we notice that the $4$-dilation of the pattern in $T_k$ with respect to $O$ gives the coloring of the subdivision of $T_{k+1}$ into triangles of side length $4$. Each of these triangles of side length $4$ is constructed from some unit triangle of $T_k$ using the rule $F_P$. 

The coloring of these triangles with side length $4$ coincides with the coloring of the corresponding segments in $P_\mathcal{S}$ because $4$-dilation of a positive (negative) triangle of $\mathcal{L}_i$ with respect to $O$ gives a positive (negative) triangle of $\mathcal{L}_{i+2}$ and they are colored identically due to Lemma \ref{lem:layer}. 

The segments inside each triangle of side length $4$ are formed by three triangles of the layer $\mathcal{L}_1$ and one triangle of the layer $\mathcal{L}_2$, see Figure \ref{pict:tr4}. The coloring of these triangles is determined by the coloring of the internal segments from the rule $F_P$ and this coloring coincides with the coloring of the positive/negative triangles in the corresponding layers from Lemma \ref{lem:layer}.

\begin{figure}[!ht]
\begin{center} 
\includegraphics[scale=0.7]{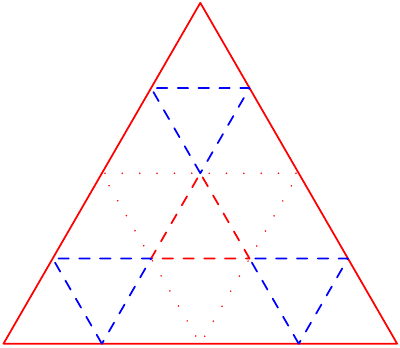}
\hskip 1cm
\includegraphics[scale=0.7]{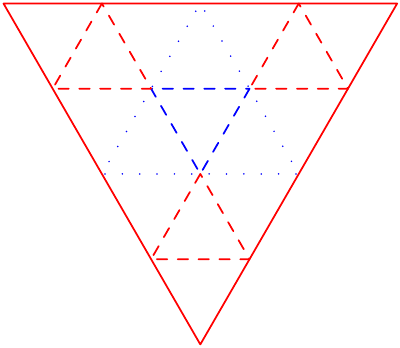}
\caption{The triangles from the layer $\mathcal{L}_1$ (dashed) and from the layer $\mathcal{L}_2$ (dotted).}
\label{pict:tr4}
\end{center}
\end{figure}
\end{proof}

As a consequence we can use Perron-Frobenius theory for substitution tilings to compute densities of all triangles in the tiling $T_\mathcal{S}$ and in the pattern $P_\mathcal{S}$.

\begin{exam}
The rule $F_P$ uses 8 different types of triangles, so the corresponding substitution matrix $M_{P}$ is an $8\times 8$ integer matrix. If we use the same order of triangles as in $F_P$ (see Figure \ref{pict:sub}), then

$$
M_{P}=
\left(
\begin{array}{rrrrrrrr}
1&1&1&1&3&3&3&3\\
9&5&2&0&0&0&0&0\\
0&4&6&6&0&0&0&0\\
0&0&1&3&3&3&3&3\\
3&3&3&3&1&1&1&1\\
0&0&0&0&9&5&2&0\\
0&0&0&0&0&4&6&6\\
3&3&3&3&0&0&1&3
\end{array}
\right)
$$
where the first 4 columns correspond to positive triangles, and the last 4 to negative. Since 
$$M_{P}^2=
\left(
\begin{array}{rrrrrrrr}
 28 & 28 & 28 & 28 & 36 & 36 & 36 & 36 \\
 54 & 42 & 31 & 21 & 27 & 27 & 27 & 27 \\
 36 & 44 & 50 & 54 & 18 & 18 & 18 & 18 \\
 18 & 22 & 27 & 33 & 39 & 39 & 39 & 39 \\
 36 & 36 & 36 & 36 & 28 & 28 & 28 & 28 \\
 27 & 27 & 27 & 27 & 54 & 42 & 31 & 21 \\
 18 & 18 & 18 & 18 & 36 & 44 & 50 & 54 \\
 39 & 39 & 39 & 39 & 18 & 22 & 27 & 33 \\
\end{array}
\right)$$
is a positive matrix, we can see that the substitution is primitive. The Perron-Frobenius eigenvalue of $M_{P}$ is $\lambda_{PF}=16$ and all other eigenvalues are $4,4,4,1,1,0,0$. 

The left and right $\lambda_{PF}$-eigenvectors of $M_{P}$ are $(1,1,1,1,1,1,1,1)$ and $(1,1,1,1,1,1,1,1)^t$ respectively. The left $PF$-eigenvector of $M_{P}$ shows that areas of all prototiles (eight triangles) can be chosen to be equal which we already know because all prototiles are equal regular triangles of different colors. The right $PF$-eigenvector shows that densities of all types of triangles are equal (to $\frac18$) in the tiling $T_\mathcal{S}$. In this case we say that two triangles are of the same type if they are translations of each other or rotations by a multiple of $\frac{2\pi}{3}$. 

Since the tiling $T_\mathcal{S}$ has a three-fold rotational symmetry with respect to the origin, then the densities of translationally different tiles in $T_\mathcal{S}$ can be found as well. Particularly, all four types of triangles without decoration have density $\frac18$, and all twelve types of triangles with decoration have density $\frac{1}{24}$.
\end{exam}

\section{Periodic sequences of elementary foldings}\label{sec:periodic}

The main goal of this section is to prove that if a sequence $\mathcal{S}$ of elementary foldings is periodic, then the corresponding pattern $P_\mathcal{S}$ and tiling $T_\mathcal{S}$ can be generated via substitution rules. In order to show that, we first introduce two auxiliary substitutions that will help us to describe substitution rules for periodic $\mathcal{S}$ as well as establish common spectral properties of the corresponding substitutions.

The substitution rules for the all-up sequence given in Figure \ref{pict:sub} can be viewed as the square of the substitution rules $F_+$ defined in Figure \ref{pict:f+}.

\begin{figure}[!ht]
\begin{center} 
\begin{tabular}{llll}
\includegraphics[scale=0.5]{basic-111.pdf}$\longrightarrow$\includegraphics[scale=0.5]{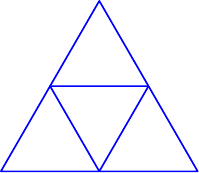}&
\includegraphics[scale=0.5]{basic-112.pdf}$\longrightarrow$\includegraphics[scale=0.5]{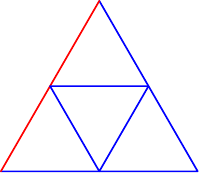}&
\includegraphics[scale=0.5]{basic-122.pdf}$\longrightarrow$\includegraphics[scale=0.5]{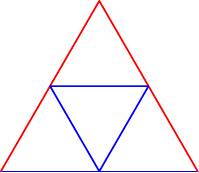}&
\includegraphics[scale=0.5]{basic-222.pdf}$\longrightarrow$\includegraphics[scale=0.5]{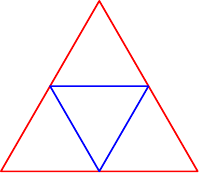}
\\[0.2cm]
\includegraphics[scale=0.5]{basic-444.pdf}$\longrightarrow$\includegraphics[scale=0.5]{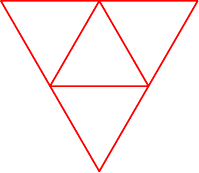}&
\includegraphics[scale=0.5]{basic-344.pdf}$\longrightarrow$\includegraphics[scale=0.5]{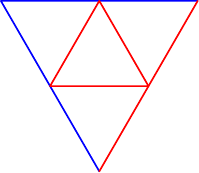}&
\includegraphics[scale=0.5]{basic-334.pdf}$\longrightarrow$\includegraphics[scale=0.5]{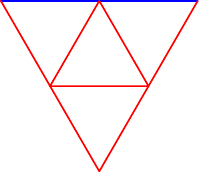}&
\includegraphics[scale=0.5]{basic-333.pdf}$\longrightarrow$\includegraphics[scale=0.5]{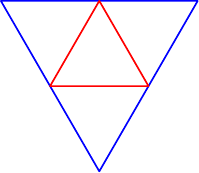}
\end{tabular}
\caption{Auxiliary substitution rules $F_+$.}
\label{pict:f+}
\end{center}
\end{figure}

Indeed, if we apply these rules to a positive or negative triangle twice, then we get exactly the substitution rules $F_P$, see Figure \ref{pict:f+squared} for illustration of this property for a positive triangle and for a negative triangle.

\begin{figure}[!ht]
\begin{center}
\begin{tabular}{l}
\includegraphics[scale=0.5]{basic-112.pdf}$\longrightarrow$\includegraphics[scale=0.5]{f+-112.pdf}
$\longrightarrow$\includegraphics[scale=0.5]{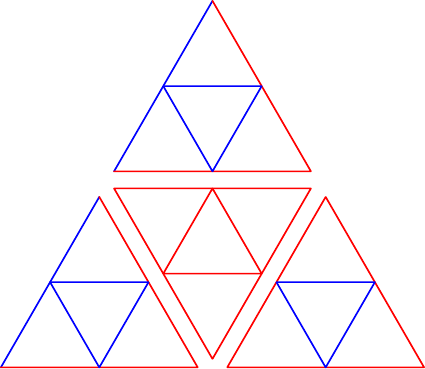}$=$\includegraphics[scale=0.5]{rules-112.pdf}
\\[0.2cm]
\includegraphics[scale=0.5]{basic-444.pdf}$\longrightarrow$\includegraphics[scale=0.5]{f+-444.pdf}
$\longrightarrow$\includegraphics[scale=0.5]{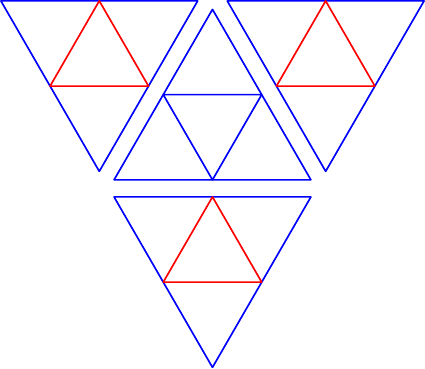}$=$\includegraphics[scale=0.5]{rules-444.pdf}
\end{tabular}
\caption{Substitution rules $F_P$ as $F_+^2$.}
\label{pict:f+squared}
\end{center}
\end{figure}

We also introduce additional auxiliary substitution rules $F_-$ (see Figure \ref{pict:f-}) that differ from $F_+$ only by the coloring of central triangles.

\begin{figure}[!ht]
\begin{center} 
\begin{tabular}{llll}
\includegraphics[scale=0.5]{basic-111.pdf}$\longrightarrow$\includegraphics[scale=0.5]{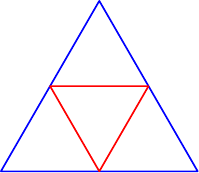}&
\includegraphics[scale=0.5]{basic-112.pdf}$\longrightarrow$\includegraphics[scale=0.5]{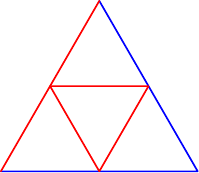}&
\includegraphics[scale=0.5]{basic-122.pdf}$\longrightarrow$\includegraphics[scale=0.5]{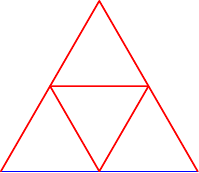}&
\includegraphics[scale=0.5]{basic-222.pdf}$\longrightarrow$\includegraphics[scale=0.5]{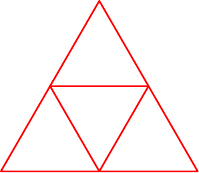}
\\[0.2cm]
\includegraphics[scale=0.5]{basic-444.pdf}$\longrightarrow$\includegraphics[scale=0.5]{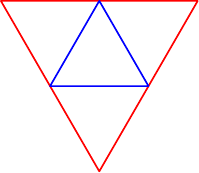}&
\includegraphics[scale=0.5]{basic-344.pdf}$\longrightarrow$\includegraphics[scale=0.5]{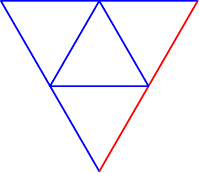}&
\includegraphics[scale=0.5]{basic-334.pdf}$\longrightarrow$\includegraphics[scale=0.5]{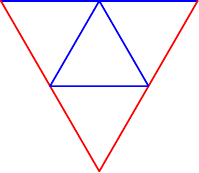}&
\includegraphics[scale=0.5]{basic-333.pdf}$\longrightarrow$\includegraphics[scale=0.5]{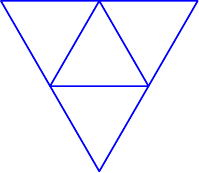}
\end{tabular}
\caption{Auxiliary substitution rules $F_-$.}
\label{pict:f-}
\end{center}
\end{figure}

\begin{lem}\label{lem:pattern=+-}
Let $\mathcal F=\{a_i\}_{i=1}^k$ be a finite sequence of elementary foldings. 

If $k$ is even, then the finite folding pattern $P_\mathcal F$ can be obtained as $$P_\mathcal F=F_{a_1}\circ F_{a_2}\circ \ldots \circ F_{a_k}(\text{\includegraphics[scale=0.4]{basic-111.pdf}})$$
for a certain choice of the boundary coloring.

Similarly, if $k$ is odd, then the finite folding pattern $P_\mathcal F$ can be obtained as $$P_\mathcal F=F_{a_1}\circ F_{a_2}\circ \ldots \circ F_{a_k}(\text{\includegraphics[scale=0.4]{basic-333.pdf}})$$
for a certain choice of the boundary coloring.
\end{lem}
\begin{proof}
First of all we notice that the composition described in the lemma leads to a triangle of side length $2^k$ which is positive if $k$ is even and negative if $k$ is odd, so the size and orientation of the resulting pattern coincides with the size and orientation of $P_\mathcal F$. In the rest of the proof we show that the colorings of the layers $\mathcal L_j$, $1\leq j\leq k-1$ in the composition and in $P_\mathcal F$ coincide too. Note, that we do not need equal coloring of the layer $\mathcal L_k$ as this layer is exactly on the boundary of $P_\mathcal F$.

Suppose $k$ is even. After we construct the pattern $\mathcal P_1:=F_{a_2}\circ \ldots \circ F_{a_k}(\text{\includegraphics[scale=0.4]{basic-111.pdf}})$, we get a pattern inside a positive triangle with side length $2^{k-1}$. According to the rules $F_+$ and $F_-$, in order to construct $P_\mathcal F=F_{a_1}(\mathcal P_1)$ we need to inflate the pattern $\mathcal P_1$ twice, swap the colors in the inflated pattern, and then add either red or blue triangles inside every resulting triangle of side length $2$. We add negative blue triangles and positive red triangles if $a_1=+$ and negative red triangles and positive blue triangles if $a_1=-$. 

Let us consider the pattern $-\mathcal P_1$. It represents a coloring of the layers $\mathcal L_j$, $1\leq j\leq k-1$ inside a negative triangle of side length $2^{k-1}$. In order to get the pattern $2\mathcal P_1$ we apply $(-2)$-homothety, so we get some coloring of the layers $\mathcal L_j$, $2\leq j\leq k$ (that needs to be swapped) and the additional red and blue triangles are exactly triangles from the missing layer $\mathcal L_1$. Note, that the coloring we get coincides with the one described in Lemma \ref{lem:layer}.

In a similar way, let $\mathcal P_2=F_{a_3}\circ \ldots \circ F_{a_k}(\text{\includegraphics[scale=0.4]{basic-111.pdf}})$. This is a pattern inside positive triangle of side length $2^{k-2}$. If we inflate this pattern with a factor of $4$ we get a coloring of all layers inside a positive triangle of side length $2^k$ except $\mathcal L_1$ and $\mathcal L_2$; note that this inflation does not necessarily give a proper coloring of these layers immediately. The layer $\mathcal L_1$ is colored during the last substitution $F_{a_1}$, so the coloring of $\mathcal L_2$ is defined by the substitution $F_{a_2}$ applied to $\mathcal P_2$ and then adjusted during the last substitution $F_{a_1}$.

In order to get the pattern $F_{a_2}(\mathcal P_2)$ we stretch $\mathcal P_2$ with the factor of $2$ and fill the resulting triangles of side length $2$ with unit red or blue triangles. After that, when we apply $F_{a_1}$, these unit triangles are inflated to triangles with side length $2$ of $\mathcal L_2$ and we swap the color of their sides according to substitution $F_+$ or $F_-$. So, if $a_2=+$, then initial positive triangles are colored with red and initial negative triangles are colored with blue. After we inflate and swap colors, the positive triangles of $\mathcal L_2$ are colored with blue and the negative triangles of $\mathcal L_2$ are colored with red. Similarly, if $a_2=-$, then the positive triangles of $\mathcal L_2$ are colored with red and the negative triangles of $\mathcal L_2$ are colored with blue. This coloring coincides with the coloring of $\mathcal L_2$ defined by $a_2$ according to Lemma \ref{lem:layer}.

In the same way, let  $\mathcal P_j:=F_{a_{j+1}}\circ \ldots \circ F_{a_k}(\text{\includegraphics[scale=0.4]{basic-111.pdf}})$. The coloring of the layer $\mathcal L_j$ in $P_\mathcal F$ is defined only by the substitution $F_{a_j}$ applied to $\mathcal P_j$ as central triangles introduced during this step form the layer $\mathcal L_1$ of the pattern $\mathcal P_{j-1}$; these triangles are inflated to triangles of the layer $\mathcal L_2$ of the pattern $\mathcal P_{j-2}$ by $F_{a_{j-1}}$, and so on. After the process of $j-1$ inflations applied to $\mathcal P_{j-1}=F_{a_j}(\mathcal P_j)$ we get triangles of $\mathcal L_{j}$ with side length $2^{j-1}$ and swap colors of their sides $j-1$ times which leads to a coloring of $\mathcal L_j$. This coloring of $\mathcal L_j$ coincides with the one of $P_\mathcal F$ due to Lemma \ref{lem:layer}.

The case of odd $k$ is similar.
\end{proof}

We also note that we can use any positive unit triangle in case of even $k$ and any negative unit triangle in case of odd $k$ as the resulting composition of substitution rules gives patterns that differ only in the coloring of the boundary.

The construction described in the previous lemma fits the approach to construct tilings using mixed substitutions. In that approach, several different substitution rules can be applied to the whole tiling, or to separate tiles, or to collections of tiles on each inflation step. Particularly, as described in Lemma \ref{lem:pattern=+-}, we can obtain finite folding patterns using substitution rules $F_+$ and $F_-$ and applying one of these rules to the existing pattern on each step based on the sequence of the elementary foldings. We refer to papers \cite{FS, GM} for a rigorous description of mixed substitution rules.

Particularly, the description from Lemma \ref{lem:pattern=+-} satisfies construction from \cite[Def. 1.2]{GM} and allows us to use all machinery of mixed substitutions immediately. If we start from a finite sequence $\mathcal F=\{a_i\}_{i=1}^k$ and create a new sequence  $\mathcal F'=\{a_i\}_{i=1}^{k+1}$ but with an additional elementary folding $a_{k+1}$, then the resulting folding patterns coincide on the common part. Similarly, in order to make the corresponding change in the composition of substitution rules from $F_{a_1}\circ F_{a_2}\circ \ldots \circ F_{a_k}$ to $F_{a_1}\circ F_{a_2}\circ \ldots \circ F_{a_k}\circ F_{a_{k+1}}$ we need to insert the new rule $F_{a_{k+1}}$ as the initial substitution as defined in \cite[Def 1.2]{GM}. We also refer to \cite{Fre} for ``unconventional'' approach to mixed substitutions where each new substitution is applied to the result of the previous substitutions rather than as the initial step.

In addition to mixed substitutions approach, the description from Lemma \ref{lem:pattern=+-} gives a way to show existence of substitution rules for the set of folding patterns generated by periodic sequences as well as study their spectral properties. While we will not use properties of mixed substitutions here, we later use them in Section \ref{sec:density} to show that densities for different types of triangles are equal in $P_\mathcal S$ for every sequence $\mathcal S$.

\begin{thm}\label{thm:periodic}
If a sequence $\mathcal S=\{a_i\}_{i=1}^\infty$ is periodic with even period $\{a_1,\ldots,a_k\}$ then the pattern $P_{\mathcal S}$ can be generated using substitution rules $F_{a_1}\circ F_{a_2}\circ \ldots \circ F_{a_k}$ applied either to \includegraphics[scale=0.4]{basic-111.pdf} if $a_1=+$ or to \includegraphics[scale=0.4]{basic-222.pdf} if $a_1=-$ as a seed.
\end{thm}
\begin{rem}
We use an even period of the sequence $\mathcal S$ in order to make existence of a legal seed for the corresponding substitution more clear. Similarly to the sequence of all foldings up with odd period 1, it is simpler to use the substitution rules $F_P$ (see Figure \ref{pict:sub}) with explicit legal seed rather than to use the rules $F_+$ (see Figure \ref{pict:f+}) without explicit seed visible.
\end{rem}
\begin{proof}
Suppose $a_1=+$ as the second case is similar. The pattern $P_\mathcal S$ can be obtained as the limit of patterns $P_{\mathcal S_{kn}}$ where $\mathcal S_{kn}=\{a_i\}_{i=1}^{kn}$. According to Lemma \ref{lem:pattern=+-}
$$P_{\mathcal S_{kn}}=(F_{a_1}\circ F_{a_2}\circ \ldots \circ F_{a_k})^n(\text{\includegraphics[scale=0.4]{basic-111.pdf}}).$$ Moreover, the central unit triangle of $P_{\mathcal S_{kn}}$ is a positive triangle of $\mathcal L_1$, so its sides are red because $a_1=+$.

Hence, $P_{\mathcal S_{kn}}$ can be obtained by using substitution rules $F_{a_1}\circ F_{a_2}\circ \ldots \circ F_{a_k}$ $n$ times to \includegraphics[scale=0.4]{basic-111.pdf} and the whole pattern $P_\mathcal S$ can be generated using substitution rules $F_{a_1}\circ F_{a_2}\circ \ldots \circ F_{a_k}$ applied to \includegraphics[scale=0.4]{basic-111.pdf}.
\end{proof}

Now we turn our attention to spectral properties of such substitutions. Particularly, existence of auxiliary substitutions $F_+$ and $F_-$ will allow us to find all eigenvalues and most eigenvectors using the corresponding matrices $M_+$ and $M_-$.

\begin{defin}\label{def:+-}
Let $M_+$ and $M_-$ denote the substitution matrices for the rules $F_+$ and $F_-$ defined in Figures \ref{pict:f+} and \ref{pict:f-} respectively. So,
$$
M_{+}=\left(
\begin{array}{rrrrrrrr}
0&0&0&0&1&1&1&1\\
0&0&1&3&0&0&0&0\\
0&2&2&0&0&0&0&0\\
3&1&0&0&0&0&0&0\\
1&1&1&1&0&0&0&0\\
0&0&0&0&0&0&1&3\\
0&0&0&0&0&2&2&0\\
0&0&0&0&3&1&0&0\\
\end{array}
\right)
\text{\qquad and \qquad}
M_{-}=\left(
\begin{array}{rrrrrrrr}
0&0&1&3&0&0&0&0\\
0&2&2&0&0&0&0&0\\
3&1&0&0&0&0&0&0\\
0&0&0&0&1&1&1&1\\
0&0&0&0&0&0&1&3\\
0&0&0&0&0&2&2&0\\
0&0&0&0&3&1&0&0\\
1&1&1&1&0&0&0&0\\
\end{array}
\right).
$$

Also, for every finite sequence $\mathcal F=\{a_i\}_{i=1}^k$ we define the corresponding substitution matrix $$M_\mathcal F=M_{a_1}\ldots M_{a_k}.$$ This is exactly the matrix of the substitution pattern for the sequence $\mathcal S$ obtained by repeating $\mathcal F$ periodically if $k$ is even. If $k$ is odd, then the substitution matrix is $M_{\mathcal F}^2$ in order to get even period used in Theorem \ref{thm:periodic}. 
\end{defin}

\begin{lem}\label{lem:eigen}
If $\mathcal F$ is a sequence of length $k$, then eigenvalues of $M_\mathcal F$ are $4^k$, $2^k$, $(-2)^k$, $(-2)^k$, $1$, $1$, $0$, and $0$.
\end{lem}
\begin{proof}
Let 
$$C=\left(
\begin{array}{rrrrrrrr}
 1 & -2 & 0 & 0 & 0 & 0 & -1 & 0 \\
 1 & 0 & -2 & 0 & 1 & 0 & 3 & 0 \\
 1 & 9 & 1 & 0 & -2 & 0 & -3 & 0 \\
 1 & -3 & 1 & 0 & 1 & 0 & 1 & 0 \\
 1 & 2 & 0 & 0 & 0 & 0 & 0 & -1 \\
 1 & 0 & 0 & -2 & 0 & 1 & 0 & 3 \\
 1 & -9 & 0 & 1 & 0 & -2 & 0 & -3 \\
 1 & 3 & 0 & 1 & 0 & 1 & 0 & 1 \\
\end{array}
\right),$$
so the columns of $C$ are eigenvectors of $M_+$. Then
$$C^{-1}M_+C=\left(
\begin{array}{rrrrrrrr}
 4 & 0 & 0 & 0 & 0 & 0 & 0 & 0 \\
 0 & 2 & 0 & 0 & 0 & 0 & 0 & 0 \\
 0 & 0 & -2 & 0 & 0 & 0 & 0 & 0 \\
 0 & 0 & 0 & -2 & 0 & 0 & 0 & 0 \\
 0 & 0 & 0 & 0 & 1 & 0 & 0 & 0 \\
 0 & 0 & 0 & 0 & 0 & 1 & 0 & 0 \\
 0 & 0 & 0 & 0 & 0 & 0 & 0 & 0 \\
 0 & 0 & 0 & 0 & 0 & 0 & 0 & 0 \\
\end{array}\right)$$
and $$C^{-1}M_-C=\left(
\begin{array}{rrrrrrrr}
 4 & 0 & 0 & 0 & 0 & 0 & 0 & 0 \\
 0 & 2 & 0 & 0 & 0 & 0 & 0 & 0 \\
 0 & -8 & -2 & 0 & 0 & 0 & 0 & 0 \\
 0 & 8 & 0 & -2 & 0 & 0 & 0 & 0 \\
 0 & 14 & 6 & 0 & 1 & 0 & 0 & 0 \\
 0 & -14 & 0 & 6 & 0 & 1 & 0 & 0 \\
 0 & -4 & -4 & 0 & -1 & 0 & 0 & 0 \\
 0 & 4 & 0 & -4 & 0 & -1 & 0 & 0 \\
\end{array}
\right).$$

Then $C^{-1}M_\mathcal FC$ is a lower triangular matrix because it is a product of $k$ matrices $C^{-1}M_+C$ or $C^{-1}M_-C$. Moreover, the diagonal entries of $C^{-1}M_{\mathcal F}C$ and hence the eigenvaules of $M_\mathcal F$ are $4^k, 2^k, (-2)^k, (-2)^k, 1, 1, 0, 0$.
\end{proof}

It is possible to find five eigenvectors of $M_\mathcal F$ with eigenvalues $4^k$, $1$, $1$, $0$ and $0$ and show existence of two more eigenvectors with eigenvalues $(-2)^k$ each. Particularly, vector $(1,1,1,1,1,1,1,1)^t$ is an eigenvector of $M_{\mathcal F}$ with eigenvalue $4^k$ because it is an eigenvector of both $M_+$ and $M_-$ with eigenvalue $4$. Similarly, vectors $(1,- 3, 3, -1, 0, 0, 0, 0)^t$ and $(0,0,0,0,1, -3, 3, -1)^t$ are eigenvectors of $M_\mathcal F$ with eigenvalue $0$ because they are eigenvectors of both $M_+$ and $M_-$ with eigenvalue $0$.

To find eigenvectors with eignevalue $1$, we let $\mathbf u_+=(0,1,-2,1,0,0,0,0)^t$ and $\mathbf u_-=(1,-2,1,0,0,0,0,0)^t$. Then $M_+\mathbf u_+=M_+\mathbf u_-=\mathbf u_+$ and $M_-\mathbf u_+=M_-\mathbf u_-=\mathbf u_-$. Therefore, if $a_1=+$, then $M_\mathcal F\mathbf u_+=\mathbf u_+$ and if $a_1=-$, then $M_\mathcal F\mathbf u_-=\mathbf u_-$, so one of these vectors is an eigenvector of $M_\mathcal F$ with eigenvalue $1$. Similarly, one of the vectors $(0, 0, 0, 0, 0, 1, -2, 1)^t$ or $(0, 0, 0, 0, 1, -2, 1, 0)^t$ is an eigenvector of $M_\mathcal F$ with eigenvalue $1$.

Also from matrices $C^{-1}M_+C$ and $C^{-1}M_-C$ we can see that the 3rd, the 5th, and the 7th column of the matrix $C$ form a three-dimensional invariant subspace for both $M_+$ and $M_-$; this subspace is given by the linear system $x_1+x_2+x_2+x_4=x_5=x_6=x_7=x_8=0$. Restriction of both $M_+$ and $M_-$ as well as the product $M_\mathcal F$ on this subspace have lower triangular matrices in the basis given by columns of $C$, therefore the restriction of $M_\mathcal F$ has eigenvalues $(-2)^k$, $1$, and $0$, and $M_\mathcal F$ has an eigenvector with eigenvalue $(-2)^k$ in this subspace. We can use similar arguments to show existence of another eigenvector of $M_\mathcal F$ with eigenvalue $(-2)^k$ in the invariant subspace $x_1=x_2=x_3=x_4=x_5+x_6+x_7+x_8=0$ spanned by the 4th, the 6th, and the 8th column of $C$. It is possible to find the coordinates of these eigenvectors explicitly using either the sequence $\mathcal F$ or the types of triangles adjacent to the sides of the triangle $$P_\mathcal F=F_{a_1}\circ F_{a_2}\circ \ldots \circ F_{a_k}(\text{\includegraphics[scale=0.4]{basic-111.pdf}}).$$

The last eigenvalue of $M_\mathcal F$ is $2^k$. If $k$ is odd, then there is an eigenvector with this eigenvalue, so $M_\mathcal F$ is diagonalizable. If $k$ is even, then $2^k=(-2)^k$ and the additional eigenvector can't be guaranteed. Moreover, if $\mathcal F=\{+,-\}$ then $M_\mathcal F$ is not diagonalizable, so the additional eigenvector does not always exist.

We conclude the section with the theorem that combines this information about eigenvalues and eigenvectors.

\begin{thm}\label{thm:eigensystem}
Let $\mathcal S$ be a periodic sequence of elementary foldings with an even period $k$. Then the matrix $M_\mathcal S$ correspoinding to the substitution rules from Theorem $\ref{thm:periodic}$ has eigenvalues $4^k, 2^k,2^k,2^k,1,1,0,0$. Moreover,
\begin{itemize}
\item $4^k$ is the PF-eigenvalue of $M_\mathcal S$ with eigenvector $(1,1,1,1,1,1,1,1)^t$ so all eight types of triangles have density $\frac18$ in the pattern $P_\mathcal S$;
\item the eigenspace of $M_\mathcal S$ with eigenvalue $1$ is spanned by $(0,1,-2,1,0,0,0,0)^t$ and $(0, 0, 0, 0, 0, 1, -2, 1)^t$ if the first element of $\mathcal S$ is $+$ and is spanned by $(1,-2,1,0,0,0,0,0)^t$ and $(0, 0, 0, 0, 1, -2, 1,0)^t$ if the first element of $\mathcal S$ is $-$;
{\sloppy

}
\item the eigenspace of $M_\mathcal S$ with eigenvalue $0$ is spanned by $(1,- 3, 3, -1, 0, 0, 0, 0)^t$ and $(0,0,0,0,1, -3, 3, -1)^t$;
{\sloppy

}
\item the eigenspace corresponding to the eigenvalue $2^k$ has dimension $2$ or $3$.
\end{itemize}
\end{thm}
\begin{proof}
It is enough to recall that from Theorem \ref{thm:periodic} the substitution matrix $M_\mathcal S$ for the rules generating $P_\mathcal F$ is a product of $k$ matrices $M_+$ and $M_-$ and then use Lemma \ref{lem:eigen} for an even $k$ to get the eigenvalues.

Then $4^k$ is the PF-eigenvalue of $M_\mathcal S$ and the PF-eigenvector of $M_\mathcal S$ is $(1,1,1,1,1,1,1,1)^t$ which means that densities of all eight types of triangles are equal. Moreover, similarly to Section \ref{sec:allup} we can use 3-fold rotational symmetry as well.

Existence of four other eigenvectors and the dimension of the last eigenspace follow from the discussion before this theorem.
\end{proof}

\section{Densities for arbitrary sequences}\label{sec:density}

The goal of this section is to prove that for every sequence $\mathcal{S}$ of elementary foldings the densities of all types of unit triangles (four types of positive colorings, and four types of negative colorings) in $P_\mathcal{S}$ are equal to $\frac18$. In case $\mathcal S$ is periodic this statement was established in Theorem \ref{thm:eigensystem} using Perron-Frobenuis theory. For an arbitrary sequence $\mathcal S$ we can use a similar theory for mixed substitutions however first we present an elementary proof.

\begin{thm}\label{thm:densities}
Let $\mathcal{S}$ be a sequence of elementary foldings. Then the densities of all types of unit triangles are equal in $P_\mathcal{S}$.
\end{thm}
\begin{proof}
Recall that both matrices $M_+$ and $M_-$ have the same Perron-Frobenius eigenvalue 4 and the same Perron-Frobenius eigenvector $\mathbf u=(1,1,1,1,1,1,1,1)^t$. Also the orthogonal complement $\mathbf u^\perp$ given by $x_1+x_2+x_3+x_4+x_5+x_6+x_7+x_8=0$ is an invariant subspace for both $M_+$ and $M_-$. Absolute values of eigenvalues of restrictions of $M_+$ and $M_-$ on $\mathbf  u^\perp$ are at most $2$.

Suppose that $\mathbf x=\mathbf y +\mathbf z$ where $\mathbf x$ is a multiple of $\mathbf u$ and $\mathbf z\in \mathbf u^\perp$. Let $\mathcal S_k$ be the sequence of first $k$ terms of $\mathcal S$. Then
$$
\lim_{k\rightarrow \infty}\frac{M_{\mathcal S_k}\mathbf x}{4^k}=\mathbf y
$$
because $M_+\mathbf y=M_-\mathbf y=4\mathbf y$, $||M_+\mathbf z||_2\leq 2||\mathbf z||_2$ and $||M_-\mathbf z||_2\leq 2||\mathbf z||_2$. Here the limit means coordinate-wise convergence.

If we start from a single triangle and apply substitutions $F_+$ and $F_-$ repeatedly, then the corresponding initial vector $\mathbf x$ that counts types of triangles has single non-zero entry equal to 1, and the corresponding limit of ratio  $\frac{M_{\mathcal S_k}\mathbf x}{4^k}$ that represent densities exists and equal to $\frac 18\mathbf u$. This means that the densities of all types of unit triangles are equal to $\frac 18$ in $P_\mathcal S$.
\end{proof}

\begin{rem}
Another approach to the densities uses mixed substitutions. The mixed substitution system generated by $\mathcal S$ using $F_+$ and $F_-$ is a primitive one. According to \cite[Sect. 2.5]{BST} (though it uses one-dimensional sequence of letters rather than two-dimensional tilings), $(1,1,1,1,1,1,1,1)^t$ is a generalized Perron-Frobenius eigenvector for the mixed substitution system defined by $\mathcal S$ and it encodes the densities of triangles types in case $\mathcal S$ is a repetitive sequence. Thus, densities of all types of triangles are equal for such $\mathcal S$.

For an arbitrary sequence we can use results of \cite[Sections 3.4 and 3.5]{FS} to show unique ergodicity of associated hulls and existence of densities that are encoded by common PF-eigenvector of $M_+$ and $M_-$. However, this also requires that our substitutions should be strictly primitive meaning that we need to combine several instances of $F_+$ and $F_-$ into one substitution step prior to using this approach.
\end{rem}

\section{Further paperfolding patterns}

We conclude the paper with several possible avenues to study paperfolding patterns. This includes other possible ways to get patterns on the grid $\mathcal L$ as well as patterns on other grids. Unfortunately, we are unable to provide much insight for most of the cases, but questions on existence of substitution rules as well as topological, spectral, and dynamical properties of all such patterns might be interesting.

\subsection{Mixed folding patterns on $\mathcal L$.} In general, we are not required to have all three parts of an elementary folding in Definition \ref{def:elementary} to be performed through one half-space of the ambient space. This will give eight options for what we will call {\it mixed elementary foldings} of a triangle with side length $2a$ assuming we can independently choose which half-space is used for every side triangle with side length $a$.

Most preliminary results of the current paper are true for mixed foldings as well, but the coloring of the layers could be different now. Particularly, Lemma \ref{lem:layer} and Theorem \ref{thm:aperiodic} hold for such mixed substitutions, but Lemma \ref{lem:hexagon} on colorings of segments incident to one vertex of $\mathcal L$ does not hold for mixed substitutions.

In particular, if a mixed elementary folding is not one of the elementary foldings introduced before, then the layer $\mathcal{L}_1$ is colored according to Figure \ref{pict:mixed} (or its rotation) if the first mixed elementary folding in a sequence defining a pattern has two operations performed through the upper half-space, and one through the lower half-space. In that case there will be vertices with coloring of six incident segments other than shown in  Figure \ref{pict:local}.

\begin{figure}[!ht]
\begin{center} 
\includegraphics[width=0.8\textwidth]{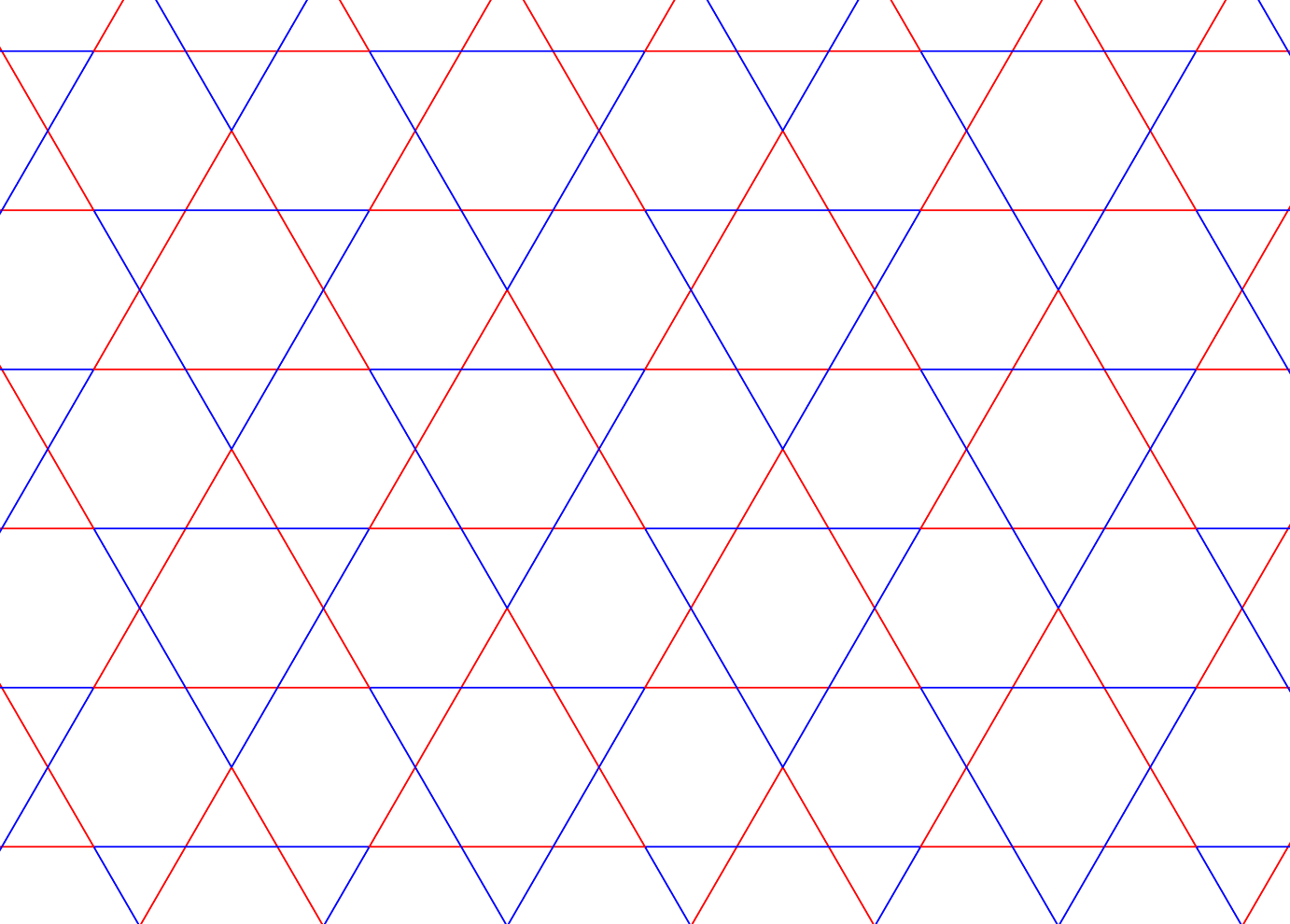}
\caption{A coloring of $\mathcal{L}_1$ for a mixed folding.}
\label{pict:mixed}
\end{center}
\end{figure}

We can not provide more details on the structure of the resulting pattern at this point. In particular, the densities of the unit triangles are no longer equal as can be seen from the layer $\mathcal{L}_1$ in Figure \ref{pict:mixed}. For example, all positive triangles from this layer have one red side and two blue sides which implies that the density of this type is at least $\frac18$. However more such triangles could appear in the hexagons of the layer $\mathcal{L}_1$.

\subsection{Other grids.} We can try to use other types of grids as our underlying structure for coloring of peaks and valleys. Particularly, in this paper we use the triangular grid $\mathcal L$, but some generalizations of the classical paperfolding sequence are using the square grid as the underlying structure, see \cite{BQS} for example. In order to use a similar approach, we can start from a shape, say a polygon, $\mathcal P$ that can be folded into a scaled copy of $\mathcal P$. While for an arbitrary $\mathcal P$ it may generate an ``irregular'' picture after repeated foldings and unfoldings (for example we fold a regular $7$-gon along seven segments connecting midpoints of adjacent sides), for certain shapes, such as a regular triangle in this paper or a square in \cite{BQS}, we can get ``nicer'' structures.

As another example of such shape, we can fold an isosceles right triangle with leg length $2$ along its shorter midsegment, then along its altitude, and finally along the second shorter midsegment. In the end we get a scaled copy of the initial triangle with leg length $1$. Assuming all foldings were done through the upper halfspace, once we unfold the triangle back, we will get the pattern given in Figure \ref{pict:right} (left).

If we perform the same procedure once again through the upper halfspace, then we get a pattern inside a triangle with leg length $4$. This pattern is also given in Figure \ref{pict:right}. Following the approach of Definition \ref{def:pattern}, we can define an infinite folding pattern using the described basic folding. This pattern will use a part of the square grid where each square is dissected with a diagonal. With a proper choice of a smaller triangle inside, we can extend this pattern onto the whole grid of dissected squares.

\begin{figure}[!ht]
\begin{center} 
\includegraphics[scale=0.5]{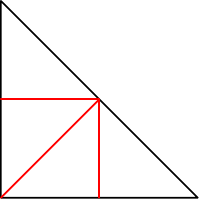}
\qquad \qquad
\includegraphics[scale=0.5]{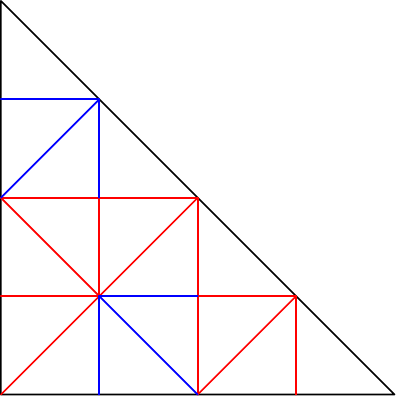}
\caption{Folding patterns inside isosceles right triangles.}
\label{pict:right}
\end{center}
\end{figure}

For this and similarly generated patterns, the questions of periodicity of the coloring as well as the underlying set itself look relevant as well. However, for the particular approach described above, the dissections of squares in the grid form a periodic pattern.

\section*{Acknowledgments}

The author is thankful to Dirk Frettl\"oh, Franz G\"ahler, and Johan Nilsson for fruitful discussions. The author is also thankful to CRC 701 of the Bielefeld University for support and hospitality.

Additionally, the author is thankful to all the authors of the ColorBrewer software \cite{CB} for the color scheme used in the colored tilings.

The author is also thankful to anonymous referees for numerous suggested improvements and particularly for pointing out to existence of substitutions $F_+$ and $F_-$ that simplified the exposition of Sections \ref{sec:periodic} and \ref{sec:density}.

\end{document}